\documentclass[10pt]{iopart}

\usepackage{cancel}
\usepackage{todonotes}
\usepackage{ulem}
\usepackage{graphicx}
\usepackage{xcolor}
\usepackage{amsmath}
\usepackage{amsthm}
\theoremstyle{plain}
\usepackage{txfonts}
\usepackage{hyperref}
\usepackage{algorithm}
\usepackage{algpseudocode}
\newcommand{\diver}{\operatorname{div}}

\newcommand{\diag}{\operatorname{diag}}
\newcommand{\Diag}{\operatorname{Diag}}
\newcommand{\biDiag}{\operatorname{\mathbf{Diag}}}
\renewcommand{\Re}{\operatorname{Re}}
\renewcommand{\Im}{\operatorname{Im}}
\newcommand{\Cov}{\operatorname{Cov}}
\newcommand{\bCov}{\operatorname{\mathbf{Cov}}}
\newcommand{\Corr}{\operatorname{Corr}}
\newcommand{\Id}{\operatorname{Id}}

\newcommand{\HS}[2]{\operatorname{HS}\left(#1,#2\right)}
\newcommand{\HSs}[1]{\operatorname{HS}\left(#1\right)}
\newcommand{\meas}{\Gamma}
\newcommand{\Bfrak}{\mathfrak{B}}
\newcommand{\R}{\mathbb{R}}
\newcommand{\C}{\mathbb{C}}
\newcommand{\Rd}{\mathbb{R}^d}
\newcommand{\E}[1]{\mathbb{E} \left[ #1 \right]}
\newtheorem{theorem}{Theorem}

\newtheorem{proposition}[theorem]{Proposition}
\newtheorem{lemma}[theorem]{Lemma}

\newtheorem{assumption}{Assumption}

\jl{5}

\begin{document}

\title[Quantitative passive imaging by iterative holography]{Quantitative passive imaging by iterative holography: The example of helioseismic holography}

\author{Björn~Müller$^1$, Thorsten~Hohage$^{1, 2}$, Damien~Fournier$^1$, Laurent~Gizon$^{1, 3}$}

\address{$^1$ Max-Planck-Institute for Solar System Research, Göttingen, Germany}
\address{$^2$ Institute for Numerical and Applied Mathematics, University of Göttingen, Germany}
\address{$^3$ Institute for Astrophysics, University of Göttingen, Germany}
\ead{muellerb@mps.mpg.de}
\vspace{10pt}
\begin{indented}
\item[]September 2023
\end{indented}

\begin{abstract}
In passive imaging, one attempts to reconstruct some coefficients in a wave equation from correlations of observed randomly excited solutions to this wave equation. Many methods proposed for this class of inverse problem so far are only qualitative, e.g., trying to identify the support of a perturbation. Major challenges are the increase in dimensionality when computing correlations from primary data in a preprocessing step, and often very poor pointwise signal-to-noise ratios. 
In this paper, we propose an approach that addresses both of these challenges: It works only on the primary data while implicitly using the full information contained in the correlation data, and it provides quantitative estimates and convergence by iteration. 

Our work is motivated by helioseismic holography, a well-established imaging method to map heterogenities and flows in the solar interior. 
We show that the back-propagation used in classical helioseismic holography can be interpreted as the adjoint of the Fr\'echet derivative of the operator which maps the properties of the solar interior to the correlation data on the solar surface. 
The theoretical and numerical framework 
for passive imaging problems developed in this paper 
extends helioseismic holography to nonlinear problems 
and allows for quantitative reconstructions. 
We present a proof of concept in uniform media.
   
\end{abstract}

%
\vspace{2pc}

\section{Introduction}
In this paper, we consider passive imaging problems 
described by a linear time-harmonic wave equation 
\[
L[q]\psi =s
\]
with a random source $s$ and some unknown coefficient 
$q$, which is the quantity of interest. We assume 
that $\E{s}=0$ such that $\E{\psi}=0$ by linearity 
of $L[q]$. Solutions $\psi$ to this wave equation 
are observed on part of the boundary $\Gamma=\partial \Omega$ of 
a domain $\Omega$ for many independent realizations 
of $s$. Thus we can approximately compute the 
cross-covariance
   \begin{eqnarray}
       \label{eq: correlation}
       C(\bi{x}_1, \bi{x}_2)=\E{\psi(\bi{x}_1) \overline{\psi(\bi{x}_2)}}, \qquad \bi{x}_1, \bi{x}_2\in\Gamma.
       \qquad 
   \end{eqnarray}
Our aim is to determine the 
unknown parameter $q$ given noisy observations of 
$C$ or the corresponding integral operator 
$(\mathcal{C}f)(\bi{x}_1):=\int_{\Gamma}
C(\bi{x}_1, \bi{x}_2)f(\bi{x}_2)\,\mathrm{d}\bi{x}_2$. 
If $\Tr_{\Gamma} $ is the trace operator onto $\Gamma$, 
then straightforward calculations show that the forward 
operator mapping $q$ to $\mathcal{C}=\bCov[\Tr_{\Gamma}  \psi]$ 
is given by
\[
\mathcal{C}[q] 
= \bCov[\Tr_{\Gamma} L[q]^{-1}s]
= \Tr_{\Gamma}  L[q]^{-1} \bCov[s] (L[q]^{-1})^{*}\Tr_{\Gamma}^*\,.
\]
(Recall that the covariance operator $\bCov[v]\in L(\mathbb{X})$ of 
a random variable $v$ with values in a Hilbert-space $\mathbb{X}$ 
is defined implicitely by 
$\Cov(\langle v,\psi\rangle_{\mathbb{X}}, \langle v,\varphi\rangle_{\mathbb{X}})
=\langle \bCov[v]\psi,\varphi\rangle_{\mathbb{X}}$ for all 
$\varphi,\psi\in\mathbb{X}$.) 
An early and influential reference on passive imaging is the work of Duvall et al.\ \cite{Duvall1993} on time-distance helioseismology. 
Later, passive imaging has also been used in many other fields such as seismology (\cite{TLHP:10}), ocean acoustics (\cite{Burov2008}), and ultrasonics (\cite{WL:01}).
We refer to the monograph \cite{Garnier2014} by Garnier \& Papanicolaou for many further references. Concerning the uniqueness of passive imaging problems, we refer to 
\cite{HLOS:18}
for results in the time domain and to 
\cite{Agaltsov2018,Agaltsov_2020,devaney:79,Hohage2020} 
for results in the frequency domain. For the unique recovery of the source and the potential from passive far-field data, we refer to \cite{Li2021}.

   Local helioseismology analyzes acoustic oscillations at the solar surface in order to reconstruct physical quantities (subsurface flows, sound speed, density) in the solar interior (e.g. \cite{Gizon2005} and references therein). Since solar oscillations are excited by near-surface turbulent convection, it is reasonable to assume random, non-deterministic noise terms. In this paper, we will describe sound propagation in the solar interior by a scalar time-harmonic wave equation and study the passive imaging problem of parameter reconstruction from correlation measurements.
   
   Very large data sets of high-resolution solar Doppler images have been recorded from the ground and from space over the last 25 years. This leads to a five-dimensional ($2^2$ spatial dimensions and $1$ temporal dimension) cross-correlation data set on the solar surface, which cannot be stored and analyzed all at once. In traditional approaches, like time-distance helioseismology, the cross-correlations are reduced to a smaller number of observable quantities, such as travel times (\cite{Duvall1993}) or cross-correlation amplitudes (e.g. \cite{Liang2013}, \cite{Nagashima2017}, \cite{Pourabdian2018}). Since the reduction to these quantities leads to a loss in information, we are interested in using the whole cross-correlation data throughout the inversion procedure, stepping forward to full waveform inversions.

   Helioseismic holography, a technique within the field of local helioseismology, has proven to be a powerful tool for studying various aspects of the Sun's interior. It operates by propagating the solar wavefield backward from the surface to specific target locations within the Sun (\cite{LB97}). A notable success of helioseismic holography is the detection of active regions on the Sun's far side (e.g. \cite{LB2000b}, \cite{Liewer2014}, \cite{Yang2023}). Furthermore, helioseismic holography is used in many other applications, e.g. to study the subsurface structure of sunspots (\cite{Braun2008}, \cite{Birch2009}, \cite{Lindsey2010}), wave absorption in magnetic regions (\cite{Cally2000}, \cite{Schunker2007}, \cite{Schunker2008}), and seismic emission from solar granules \cite{Lindsey2013}. The main idea of helioseismic holography is the back-propagation (``egression'') of the wavefield at the solar surface (\cite{LB2000}). Improvements have been proposed in the choice of backward propagators (e.g. using Porter-Bojarski holograms (\cite{Porter1982}, \cite{Gizon2018}). Helioseismic holography has a strong connection to conventional beam forming, where imaging functionals similar to the holographic back-propagation occur (e.g. \cite{Garnier2016}). In contrast to these approaches, we will achieve improvements by iterations.
   
   In the present paper, we connect holographic imaging methods to iterative regularization methods. This way, holography can be extended to a full converging regularization method. This approach was successfully applied to inverse source problems in aeroacoustics (\cite{Hohage2020}) and is extended in this work to parameter identification problems. 

   The organization of the paper is as follows. In Section~\ref{sec: model problem} we introduce a generic model for the forward problem. 
   In Section \ref{sec: diagonals} we establish foundations of our 
   functional analytic setting by establishing sufficient conditions 
   under which the diagonal of an integral operator is well defined, 
   using Schatten class properties of embedding operators. 
   With this we compute the Fr\'echet derivative of the forward 
   operator and its adjoint in section \ref{sec: frechet derivative}. 
   Next, we discuss the algorithm of iterative holography in Section~\ref{sec: Algorithm}. Based on the analysis of 
   Sections~\ref{sec: model problem}--\ref{sec: Uniform Green} we then introduce forward operators in local helioseismology, their derivatives and adjoints in Section \ref{sec: Helmholtz}. Then we discuss 
   iterative helioseismic holography as an extension of 
   conventional helioseismic holography in Section \ref{sec: applications holography}, and demonstrate its performance in numerical examples 
   with simulated data in Section~\ref{sec: Inversions} before we 
   end the paper with conclusions in Section \ref{sec: conclusions}.
   Some technical issues are discussed in three short appendices.

\section{A model problem}
\label{sec: model problem}
We first present the main ideas of this paper for a generic scalar 
time-harmonic wave equation. Let $\Omega_0\subset \Omega$ be a smooth, bounded domain in $\Rd$ and let $\meas \subset \overline{\Omega}\setminus \Omega_0$ 
the hypersurface on which measurements are performed. $\meas$ may be part 
of the boundary $\partial\Omega$ or it may be contained in the interior of
$\Omega$. Moreover, consider the parameters 
\begin{eqnarray*}
    v \in L^{\infty} (\Omega, \C), \; 
    \bi{A} \in W^{\infty}(\diver, \Omega). 
\end{eqnarray*}
Here $W^{\infty}(\diver, \Omega):=\{\bi{A}\in L^{\infty}(\Omega, \R^d):\diver \bi{A}\in L^{\infty}(\Omega)\}$ with norm
$\|\bi{A}\|_{W^{\infty}(\diver,\Omega)}:
= \|\bi{A}\|_{L^\infty}+\|\diver\bi{A}\|_{L^\infty}$. 

Assume that the excitation of wavefields $\psi$ in $\Rd$ 
by random sources $s$, which are supported in $\Omega_0$, is described by the model 
\begin{subequations}\label{eqs:forward}
\begin{align}
&(-\Delta-2i\bi{A} \cdot \nabla+v -k^2) \psi=s,&&\mbox{in }\Omega\\
\label{eq: boundary condition}
&\frac{\partial\psi}{\partial\bi{n}} = B\operatorname{Tr}_{\partial\Omega} \psi&&\mbox{on }\partial\Omega
\end{align}
\end{subequations}
for the outward pointing normal vector $\bi{n}$ on $\partial\Omega$
and some operator $B\in L\left(H^{1/2}(\partial\Omega)\to H^{-1/2}(\partial\Omega)\right)$. (Here and in the following 
$L\left(\mathbb{X},\mathbb{Y}\right)$ denotes the space of bounded linear operators 
between Banach spaces $\mathbb{X}$ and $\mathbb{Y}$.) 
Typically, $B$ is some transparent boundary condition, e.g.\ 
$B\psi =ik\psi$ for $\partial\Omega=S^{d-1}$. We may also choose $B\psi:= \operatorname{DtN}\psi$ with an exterior Dirichlet-to-Neumann 
map for the Helmholtz equation with the Sommerfeld radiation condition. 
In this case, equation \eqref{eqs:forward} is equivalent to a problem posed 
on $\Rd$ with the Sommerfeld radiation condition. 

\begin{assumption}\label{ass:well_posed}
Suppose that for some $B_0\in L\left(H^{1/2}(\partial\Omega),H^{-1/2}(\partial\Omega)\right)$, $k\in\C$ and some set $\Bfrak_k\subset 
L^{\infty} (\Omega, \C) \times W^{\infty}(\diver, \Omega)$ 
of admissible parameters $v,\bi{A}$ the following holds true:
\begin{subequations}\label{eqs:well_posed}
\begin{align}
\label{eq:assPositivity}
&\diver \bi{A}-\Im k^2 +\Im v \leq 0&&\text{in }\Omega\\
\label{eq:assnoflux}
&\bi A\cdot \bi{n}=0&&\text{on }\partial\Omega\\
\label{eq:assImB}
&\Im \int_{\partial\Omega} (B\zeta)\, \overline{\zeta}\,\mathrm{d}s >0 
&&\mbox{for all }\zeta\in H^{1/2}(\partial\Omega), \zeta\neq 0\\
\label{eq:assB0}
&\Re\int_{\partial\Omega} (B_0\zeta)\overline{\zeta}\,\mathrm{d}s \leq 0 
&&\mbox{for all }\zeta\in H^{1/2}(\partial\Omega)\\
\label{eq:assBcompact}
&B-B_0: H^{1/2}(\partial\Omega)\to H^{-1/2}(\partial\Omega)
&&\text{is compact.}
\end{align}
\end{subequations}
\end{assumption}

The conditions~\eqref{eq:assImB}--\eqref{eq:assBcompact} are obviously 
satisfied for $B\zeta:=ik\zeta$, and they also hold true if $B$ is the 
exterior Dirichlet-to-Neumann map on a sphere or a circle (see \cite{Colton2013,Ihlenburg1998}). 
Throughout this paper we denote by $H^s_0(\Omega)$ the closure of the space of distributions on $\Omega$ in $H^s(\R^d)$. For a Lipschitz domain, we have the duality $H^{s}(\Omega)^*=H^{-s}_0(\Omega)$ (\cite[Thm 3.30]{McLean2000}).

\begin{proposition}\label{prop:wellposed}
Under Assumption~\ref{ass:well_posed} 
the problem~\eqref{eqs:forward} is well posed in the 
sense that for all $s\in H^{-1}_0(\Omega)$ there exists a 
unique $\psi\in H^1(\Omega)$ satisfying \eqref{eqs:forward} in the 
weak sense, 
and $\psi$ depends continuously on $s$ with respect to these norms. 
\end{proposition}

\begin{proof}
We only sketch the proof, which is a straightforward modification of similar proofs in \cite{Colton2013,Ihlenburg1998}. The weak formulation of Problem~\eqref{eqs:forward} is given by 
\begin{align}\label{eq:weak_formulation}
\int_{\Omega}\left(\nabla\psi\cdot \nabla\overline{\phi}
-2i \bi{A} \cdot (\nabla\psi)\overline{\phi}
+(v-k^2)\psi\overline{\phi}
\right)\,\mathrm{d}\bi{x}
- \int_{\partial\Omega} B\Tr_{\partial \Omega}\psi\,\Tr_{\partial \Omega}\overline{\phi}\,\mathrm{d}s 
= \int_{\Omega}s\overline{\phi}\,\mathrm{d}\bi{x},
\,\,\phi\in H^1(\Omega).
\end{align}
To show that for $s=0$ this variational problem only has the trivial solution, we choose $\phi=\psi$ and take the imaginary part. 
Noting that ${\Im(-2i\bi{A} \cdot (\nabla\psi)\overline{\psi}) = -\bi{A} \cdot \,2\Re((\nabla\psi)\overline{\psi})= -\bi{A} \cdot \nabla|\psi|^2}$ and using 
a partial integration and \eqref{eq:assnoflux}, we obtain
\[
\int_{\Omega}\left(\operatorname{div} \bi{A}+\Im(v-k^2)\right) |\psi|^2\,\mathrm{d}\bi{x} 
= \Im\int_{\partial\Omega} B\Tr_{\partial \Omega} \psi\,\Tr_{\partial \Omega}\overline{\psi}\,\mathrm{d}\bi{s}\,.
\]
It follows from \eqref{eq:assPositivity} and \eqref{eq:assImB} that both sides must vanish. Hence, $\Tr_{\partial\Omega}\psi = 0$. By elliptic regularity, $\psi\in H^2(\Omega)$ is also a strong solution to \eqref{eqs:forward} with $\frac{\partial\psi}{\partial\bi{n}}=0$ on $\partial\Omega$. 
Due to vanishing Cauchy data on $\partial \Omega$, $\psi$ may be extended by $0$ as a strong solution of the wave equation to the exterior of $\Omega$. Now it follows from unique continuation results (see \cite[Theorem 4.2]{RL:12}) that $\psi$ vanishes identically. 

Using Assumptions~\eqref{eq:assB0} and \eqref{eq:assBcompact}, it can be shown that the sesquilinear form of the variational formulation is coercive up to a compact perturbation. Therefore, the operator representing this sesquilinear form is Fredholm of index $0$. By uniqueness, it is boundedly invertible. 
\end{proof}

If we write the solution operator 
\[
\mathcal{G}_{v,\bi{A}}:H^{-1}_0(\Omega)\to H^1(\Omega),\qquad 
\mathcal{G}_{v,\bi{A}}s:=\psi
\] 
as an integral operator 
\[
(\mathcal{G}_{v,\bi{A}}s)(\bi{x})=\int_{\Omega} G_{v,\bi{A}}(\bi{x}, \bi{y}) s(\bi{y}) \,\mathrm{d}\bi{y},
\]
the kernel $G_{v,\bi{A}}$ of $\mathcal{G}_{v,\bi{A}}$ is the Green's function, which may also be characterized by 
$(-\Delta-2i\bi{A} \cdot \nabla+v -k^2) G_{v,\bi{A}}(\cdot, \bi{x^{\prime}})
=\delta_{\bi{x^{\prime}}}$, 
$\partial_{\bi{n}} G_{v,\bi{A}}(\cdot, \bi{x^{\prime}})-B \Tr_{\partial \Omega} G_{v,\bi{A}}(\cdot, \bi{x^{\prime}})=0$ on $\partial \Omega$.

For certain random processes of interest, $s$ does not belong to $H^{-1}_0(\Omega)$ almost surely. E.g., white noise is in $H^{-s}_0(\Omega)$ almost surely if and only if $s>d/2$. Nevertheless, the solution formula 
\begin{align}\label{eq:Trpsi_by_G}
(\Tr_{\Gamma} \psi)(\bi{x}) = \int_{\Omega_0}G_{v,\bi{A}}(\bi{x},\bi{y})
s(\bi{y})\,\mathrm{d}\bi{y},\qquad \bi{x}\in \Gamma
\end{align}
may still make sense if $G_{v,\bi{A}}(\bi{x},\cdot)$ is sufficiently smooth on $\Omega_0$. This is always the case if the support of $s$ and $v,\bi{A}$ are disjoint or if $s\in H^{-1}_0(\Omega)$ almost surely, which is typically true if $s$ is spatially correlated. Otherwise, we have to impose smoothness conditions on $v$ and $\bi{A}$ such that $G_{v,\bi{A}}$ is sufficiently smooth and $\mathcal{G}$ has suitable 
mapping properties. 

\begin{assumption}\label{ass:solvability_noise}
The solution to~\eqref{eqs:forward} on $\Gamma$ is given by \eqref{eq:Trpsi_by_G}.
\end{assumption} 

Assume we have observations $\Tr_{\Gamma}\psi_1,\dots,\Tr_{\Gamma}\psi_N$ where 
$\psi_j$ solves \eqref{eqs:forward} for 
independent samples $s_1,\dots,s_N$ of $s$. As $\E{s}=0$, we have 
$\E{\Tr_{\Gamma}\psi_j}=0$, and we can compute the correlations by
\begin{align}
\label{eq: Correlation}
\Corr(\bi{x}_1,\bi{x}_2):=\frac{1}{N}\sum_{n=1}^N \Tr_{\Gamma}\psi_n(\bi{x}_1)
\overline{\Tr_{\Gamma}\psi_n(\bi{x}_2)}, \qquad \bi{x}_1, \bi{x}_2\in\Gamma.
\end{align}
This is an unbiased estimator of the covariance 
\begin{align}
\label{eq: Covariance}
C_{v,\bi{A}}(\bi{x}_1,\bi{x}_2):=\operatorname{Cov}
\left(\Tr_{\Gamma}\psi(\bi{x}_1),\Tr_{\Gamma}\psi(\bi{x}_2)\right) 
= \E{\Tr_{\Gamma}\psi(\bi{x}_1)
\overline{\Tr_{\Gamma}\psi(\bi{x}_2)}}
\end{align}
converging in the limit $N\to\infty$. The integral operator 
$(\mathcal{C}[v,\bi{A}]f)(\bi{x}_1):=\int_{\Gamma} C_{v,\bi{A}}(\bi{x}_1,\bi{x}_2)f(\bi{x}_2)\,\mathrm{d}\bi{x}_2$ is the covariance operator 
\begin{align}\label{eq:forward_op}
\mathcal{C}[v,\bi{A}] = \bCov[\Tr_{\Gamma} \mathcal{G}_{v,\bi{A}}s] 
= \Tr_{\Gamma} \mathcal{G}_{v,\bi{A}} \bCov[s]\mathcal{G}_{v,\bi{A}}^*\Tr_{\Gamma}^*.
\end{align}

$\mathcal{C}$ will be the forward operator of our inverse problem. 
Recall that if $C_{v,\bi{A}}\in L^2(\Gamma\times \Gamma)$, then 
$\mathcal{C}[v,\bi{A}]$ belongs to the space of Hilbert-Schmidt operators 
$\HSs{L^2(\Gamma)}$ on $L^2(\Gamma)$, and 
\[
\left\|\mathcal{C}[v,\bi{A}]\right\|_{\mathrm{HS}}=\|C_{v,\bi{A}}\|_{L^2}.
\]

Therefore, $\HSs{L^2(\Gamma)}$ is the natural image space of the forward 
operator. It is a Hilbert space with an inner product 
$\langle T,S\rangle_{\mathrm{HS}}= \mathrm{tr}(S^*T)$. Here $\mathrm{tr}(K)$ denotes the trace 
of a  linear operator $K: \mathbb{H} \to \mathbb{H}$ in a separable Hilbert space $\mathbb{H}$ 
defined by $\mathrm{tr}(K):=\sum_{j=1}^{\infty} \langle K e_j, e_j \rangle_{\mathbb{H}}$ for any  is an orthonormal basis 
$\{e_j:j\in\mathbb{N}\}$  of $\mathbb{H}$. 

Let us also consider the case that in addition to sources 
$s$ in the interior of $\Omega$ there are sources 
$s_{\partial\Omega}$ on the boundary $\partial\Omega$. 
Such sources generate a field 
$\psi(\bi{x}) =  \int_{\partial \Omega} G_{v,\bi{A}}(\bi{x},\bi{y})
s_{\partial\Omega}(\bi{y})\,\mathrm{d}\bi{y}$. 
Its restriction to $\Omega_0$ is given by 
\begin{align}\label{eq:SLP}
(K_{v,\bi{A}} s_{\partial\Omega}) (\bi{x}):=
(\Tr_{\Gamma} \psi)(\bi{x})
= \int_{\partial \Omega} G_{v,\bi{A}}(\bi{x},\bi{y})
s_{\partial\Omega}(\bi{y})\,\mathrm{d}\bi{y},\qquad \bi{x}\in \Gamma, 
\end{align}
which is the single layer potential operator for 
$\Gamma=\partial\Omega$. It is easy to see that 
$K$ admits a factorization 
$K=\Tr_{\Gamma} G_{v, \bi{A}} \Tr_{\partial\Omega}^*$
in the spaces
 \begin{align*}
H^{-1/2}(\partial \Omega) \xrightarrow{\Tr_{\partial\Omega}^*} H^{-1}_0(\Omega)\xrightarrow{G_{v, \bi{A}}} H^1(\Omega) \xrightarrow{\Tr_{\Gamma}} H^{1/2}(\Gamma),
\end{align*}
which implies that 
$K\in L\left(H^{-1/2}(\partial \Omega), H^{1/2}(\Gamma)\right)$ (see \cite[Thm. 1(iii)]{Costabel1988}). 
Therefore, in the presence of boundary sources the measured
covariance operator is given by 
\begin{align*}
\mathcal{C}[v,\bi{A}] 
&= \Tr_{\Gamma} \mathcal{G}_{v,\bi{A}} \bCov[s]\mathcal{G}_{v,\bi{A}}^*\Tr_{\Gamma}^*+K_{v,\bi{A}} \bCov[s_{\partial\Omega}] K^*_{v,\bi{A}}\\
&=\Tr_{\Gamma} \mathcal{G}_{v,\bi{A}} \left(\bCov[s]+
\Tr_{\partial\Omega}^*\bCov[s_{\partial\Omega}]\Tr_{\partial\Omega}\right)
\mathcal{G}_{v,\bi{A}}^*\Tr_{\Gamma}^*.
\end{align*}

Often one assumes that the source process $s$ is spatially uncorrelated and 
\[
\bCov[s]=M_{S}
\] 
for some source strength $S\in L^{\infty}(\Omega_0)$, where $M_S$ denotes 
the multiplication operator $M_{S}f:=S\cdot f$. If $S$ is treated as an additional unknown, the forward operator becomes 
\begin{align}\label{eq:extended_forward_op}
\mathcal{C}[v,\bi{A},S] 
= \Tr_{\Gamma} \mathcal{G}_{v,\bi{A}} \left(M_{S}+\Tr^*_{\partial \Omega} \bCov[s_{\partial\Omega}] \Tr_{\partial \Omega} \right) \mathcal{G}_{v,\bi{A}}^*\Tr_{\Gamma}^*.
\end{align}
Of course, we could also assume that $s_{\partial\Omega}$ is spatially uncorrelated and treat its source strength as a further unknown, but for the sake of notational simplicity, we assume that $\bCov[s_{\partial\Omega}]\in L\left(L^2(\partial\Omega)\right)$ is known.

We first study the continuity and Fr\'echet differentiability of $\mathcal{G}_{v,\bi{A}}$ with respect to the parameters $(v,\bi{A})$. We will assume that $v$ and $\bi{A}$ are known in $\Omega\setminus \Omega_0$. Let $(v_{\mathrm{ref}},A_{\mathrm{ref}})\in \mathfrak{B}_k$ be some reference solution. Then the set $\mathfrak{B}_k$ of admissible parameters in Assumption~\ref{ass:well_posed} satisfies 
\begin{align}\label{eq:defi_XG}
&\mathfrak{B}_k \subset (v_{\mathrm{ref}},A_{\mathrm{ref}}) + \mathbb{X}_{\mathcal{G}}\quad \mbox{with}\quad 
\mathbb{X}_{\mathcal{G}}:= L^{\infty}(\Omega_0)\times W^{\infty}_0(\diver,\Omega_0),
\end{align}
where $W^{\infty}_0(\diver,\Omega_0):=
\{\bi{A}\in W^{\infty}_0(\diver,\Omega_0): 
\bi{A}\cdot\bi{n}=0\text{ on }\partial\Omega_0\}$. 
\begin{lemma}\label{lem:Frechet_G}
Under Assumption \ref{ass:well_posed}, the mapping $\Bfrak_k\to L\left(H^{-1}_0(\Omega), H^1(\Omega)\right)$, $(v,\bi{A})\mapsto \mathcal{G}_{v,\bi{A}}$ is well-defined and continuous, and Fr\'echet differentiable in the interior of $\Bfrak_k$ w.r.t.\ the $\mathbb{X}_{\mathcal{G}}$-topology. The Fr\'echet derivative 
$\mathcal{G}^{\prime}_{v,\bi{A}}:\mathbb{X}_{\mathcal{G}}\to L\left(H^{-1}_0(\Omega), H^1(\Omega)\right)$ at $(v,\bi{A})\in \operatorname{int}(\Bfrak_k)$
is given by 
\[
\mathcal{G}^{\prime}_{v,\bi{A}}(\partial v,\partial \bi{A})
= \mathcal{G}_{v,\bi{A}}(2i M_{\partial \bi{A}}\cdot\nabla -M_{\partial v})\mathcal{G}_{v,\bi{A}}\,.
\]
\end{lemma}

\begin{proof}
Again, we only sketch the proof and refer to 
\cite[\S 5.3]{Colton2013} for a more detailed proof of a 
similar result. 
Let $L_{v,\bi{A}}: H^1(\Omega)\to H^{-1}_0(\Omega)$ denote the operator 
associated to the sesquilinear form in the weak formulation~\eqref{eq:weak_formulation} such 
that $G_{v,\bi{A}}=L_{v,\bi{A}}^{-1}$. 
$L_{v,\bi{A}}$ is continuous and affine linear in 
the parameters. 
As $L^{\prime}_{v,\bi{A}}(\partial v,\partial\bi{A})
=-2iM_{\partial\bi{A}}\cdot\nabla+M_{\partial v}$, the result follows from the continuity of operator inversion and the formula for its derivative, 
$\mathcal{G}^{\prime}_{v,\bi{A}}(\partial v,\partial A) = 
-\mathcal{G}_{v,\bi{A}}L^{\prime}_{v,\bi{A}}(\partial v,\partial\bi{A})\mathcal{G}_{v,\bi{A}}$.
\end{proof}

\section{Diagonals of operator kernels}
\label{sec: diagonals}
The present section serves as a preparation for computing adjoints of the Fr\'echet derivative of the forward operator defined by \eqref{eq:extended_forward_op}. A crucial step will be the characterization of adjoints of the mapping 
\[
S\mapsto M_S
\]
(in a sense to be specified later). 

In the discrete setting, $M_{S}$ corresponds to diagonal matrices $\diag(S)\in \C^{d\times d}$ with diagonal $S$. The adjoint of the mapping 
\[
\mathcal{M}: \mathbb{C}^d \rightarrow \mathbb{C}^{d \times d},
\qquad S\mapsto \diag(S)
\]
with respect to the Frobenius norm is given by 
\[
\Diag A=\Diag(A),
\]
where $\Diag(A)\in \C^d$ denotes the diagonal of the matrix $A\in \C^{d\times d}$. 

We wish to generalize this to an infinite dimensional setting, with the Frobenius norm replaced by the Hilbert-Schmidt norm. 
Recall that any operator $\mathcal{A}\in \HSs{L^2(\Omega)}$ has a Schwartz kernel $A\in L^2(\Omega\times \Omega)$ such that $(\mathcal{A}\varphi)(\bi{x})=\int_{\Omega}A(\bi{x},\bi{y})\varphi(\bi{y})\,\mathrm{d}y$ and $\|\mathcal{A}\|_{\operatorname{HS}}=\|A\|_{L^2}$. It is tempting to define $(\Diag\mathcal{A})(\bi{x}):=A(\bi{x},\bi{x})$. However, as $A$ is only a $L^2$-function and the diagonal $\{(\bi{x},\bi{x}):\bi{x}\in\Omega\}\subset \Omega\times \Omega$ has measure zero, the restriction of $A$ to the diagonal is not well-defined. 

To address this problem, we first recall that for Hilbert spaces $\mathbb{X}$, $\mathbb{Y}$ and $p\in [1,\infty)$ the $p$-Schatten class $S_p\left(\mathbb{X}, \mathbb{Y}\right)$ consists of all compact operator $\mathcal{A}\in L\left(\mathbb{X},\mathbb{Y}\right)$ for which the singular values $\sigma_j(A)$ (counted with multiplicity) form a $\ell^p$ sequence. $S_p\left(\mathbb{X}, \mathbb{Y}\right)$ is a Banach space equipped with the norm $\|\mathcal{A}\|_{S_p}:=(\sum_j \sigma_j(\mathcal{A})^p)^{1/p}$. $S_2(\mathbb{X},\mathbb{Y})$ coincides with $\HS{\mathbb{X}}{\mathbb{Y}}$. We write $S_p\left(\mathbb{X}\right):=S_p\left(\mathbb{X},\mathbb{X}\right)$. 
The elements of $S_1(\mathbb{X})$ are called trace class operators. For such operators, the trace $\tr(\mathcal{A}):=\sum_k \langle \mathcal{A} e_j,e_j\rangle$ 
is well-defined for any orthonormal basis $\{e_k\}$ of $\mathbb{X}$, and $|\tr(\mathcal{A})|\leq \|\mathcal{A}\|_{S_1}$.

Let us first recall Mercer's theorem: It states that for a positive definite operator $\mathcal{A}$ with continuous kernel $A$, we have
\[
\tr{\mathcal{A}} = \int_{\Omega}A(\bi{x},\bi{x})\,\mathrm{d} \bi{x}
\]
and $A(\bi{x},\bi{x})\geq 0$ for all $\bi{x}$. Since not all (positive semidefinite) Hilbert-Schmidt operators are trace class, we cannot expect that $\bi{x}\mapsto A(\bi{x},\bi{x})$ belongs to $L^1(\Omega)$ for general Hilbert-Schmidt operators. 
However, with the help of Mercer's theorem, we can show the following result. 

\begin{proposition}
\label{prop:diagonal}
Let $\Omega\subset\mathbb{R}^d$ be open and non-empty. Then there 
exists a unique bounded linear operator 
\[
\Diag: S_1\left(L^2(\Omega)\right)\to L^1(\Omega)
\]
such that 
\[
\Diag(\mathcal{A})(\bi{x}) = A(\bi{x},\bi{x}),\qquad \bi{x}\in\Omega
\]
for all operators $\mathcal{A}\in S_1(L^2(\Omega))$ with continuous kernel $A$. Moreover, 
\begin{align}\label{eq:trace_formula}
\tr(\mathcal{A}) = \int_{\Omega} \Diag(\mathcal{A})\,\mathrm{d} \bi{x}\,.
\end{align}
\end{proposition}

Eq.~\eqref{eq:trace_formula} is shown in \cite[Thm.~3.5]{brislawn:88} where it is also shown that $\Diag(\mathcal{A})$ can be constructed by local averaging, but the first part is not explicitly stated. We sketch an alternative, more elementary proof:

\begin{proof}[Proof of Proposition \ref{prop:diagonal}]
If $\mathcal{A}$ is positive semidefinite, it may be factorized as $\mathcal{A}=\mathcal{B}^*\mathcal{B}$ with $\mathcal{B}\in \HSs{L^2(\Omega)}$ and $\|\mathcal{A}\|_{S_1}=\|\mathcal{B}\|_{S_2}^2$, e.g., by choosing $\mathcal{B}=\mathcal{A}^{1/2}$. 
By density of $C(\overline{\Omega}\times \overline{\Omega})$ in $L^2(\Omega\times \Omega)$, there exists a sequence $(B_n)$ converging to $B$ in $L^2(\Omega\times\Omega)$. For the corresponding operators $\mathcal{B}_n$ it follows that 
$\lim_{n\to\infty}\|\mathcal{B}_n-\mathcal{B}\|_{\mathrm{HS}}=0$ and $\lim_{n\to\infty}\|\mathcal{A}_n-\mathcal{A}\|_{S_1}=0$ for $\mathcal{A}_n:=\mathcal{B}_n^*\mathcal{B}_n$ (see Prop.~\ref{prop: p-Schatten-embeddings}, part~\ref{it:Schatten_composition} below). Thus, we have constructed a sequence of positive semidefinite operators with continuous kernels converging to $\mathcal{A}$ in $S_1\left(L^2(\Omega)\right)$, and the statement follows from the classical Mercer theorem. 

We decompose a general $\mathcal{A} \in S_1\left(L^2(\Omega)\right)$ as linear combination of trace class operators: 
We start with $\mathcal{A}=\Re(\mathcal{A})+i \Im(\mathcal{A})$ where 
$\Re(\mathcal{A}):=\frac{1}{2}(\mathcal{A}+\mathcal{A}^*)$
and
$\Im(\mathcal{A}):=\frac{1}{2i}(\mathcal{A}-\mathcal{A}^*)$. 
There exists an expansion $\Re(\mathcal{A})=\sum_{k=1}^{\infty} \lambda_k \psi_k \otimes \psi_k$. We define $P_1:=\sum_{k=1}^{\infty} \operatorname{max}(\lambda_k, 0) \psi_k \otimes \psi_k, P_2:=\sum_{k=1}^{\infty} \operatorname{max}(-\lambda_k, 0) \psi_k \otimes \psi_k$ such that $\Re(\mathcal{A})=P_1-P_2$ with positive semidefinite $P_1, P_2 \in S_1\left(L^2(\Omega_0)\right)$. Therefore, a general $\mathcal{A}\in S_1\left(L^2(\Omega)\right)$ can be written as a linear combination of positive semi-definite trace class operators: 
$\mathcal{A} = \mathcal{P}_1- \mathcal{P}_2 + i\mathcal{P}_3 - i\mathcal{P}_4$ where $\Vert P_1 \Vert_{S_1},\Vert P_2 \Vert_{S_1} \leq \Vert \Re (\mathcal{A}) \Vert_{S_1}, \Vert P_3 \Vert_{S_1},\Vert P_4 \Vert_{S_1} \leq \Vert \Im (\mathcal{A}) \Vert_{S_1}$.
By the Courant-Fischer characterization 
$\sigma_n(\mathcal{A})=\inf\{\|\mathcal{A}-\mathcal{F}\|:\operatorname{rank}(\mathcal{F})\leq n\}$, we get $\sigma_{2n}(\Re(\mathcal{A})),
\sigma_{2n}(\Im(\mathcal{A}))\leq \sigma_n(\mathcal{A})$ and hence 
$\|\Re(\mathcal{A})\|_{S_1},
\|\Im(\mathcal{A})\|_{S_1}\leq 2\|\mathcal{A}\|_{S_1}$. It follows that $\|\mathcal{P}_j\|_{S_1}\leq 2\|\mathcal{A}\|_{S_1}$. 
Now we can apply the first proven special case to 
all $\mathcal{P}_j$ to obtain the result. 
\end{proof}

To speak of an adjoint of the operator $\mathcal{M}:S \mapsto M_S$, we have to treat $M_S$ in some space with a dual pairing. We will use 
Hilbert-Schmidt spaces between suitable Sobolev spaces. (Recall that $M_S:L^2(\Omega)\to L^2(\Omega)$ is not compact in general.)

Note that a Gelfand triple $\mathbb{V}^{\prime}\hookrightarrow 
\mathbb{H}\hookrightarrow \mathbb{V}$ 
of Hilbert spaces induces Gelfand triple 
\[
\HS{\mathbb{V}}{\mathbb{V}^{\prime}}\hookrightarrow \HS{\mathbb{H}}{\mathbb{H}} \hookrightarrow \HS{\mathbb{V}^{\prime}}{\mathbb{V}}
\]
of Hilbert-Schmidt spaces with dual pairing, given by 
$\langle A,B\rangle_{\operatorname{HS}} := \tr(B^*A)$ for 
$A\in \HS{\mathbb{V}}{\mathbb{V}^{\prime}}$ and $B\in\HS{\mathbb{V}^{\prime}}{\mathbb{V}}$. 

We give some preliminary results on $p$-Schatten class embeddings. 

\begin{proposition}
    \label{prop: p-Schatten-embeddings}
    Let $\Omega \subset \R^d$ be a bounded Lipschitz domain and let $\mathbb{X}, \mathbb{Y}, \mathbb{Z}$ be Hilbert spaces. Then the following holds true:
    \begin{enumerate}
        \item\label{it:Schatten_embedding} The Sobolev embedding: $j: H^m(\Omega) \hookrightarrow H^l(\Omega)$ is an element in the Schatten class $S_p\left(H^m(\Omega), H^l(\Omega)\right)$ if and only if $p > \frac{d}{m-l}$.
        \item\label{it:Schatten_composition} 
        Let $p, q, r>0$ satisfy $\frac{1}{p}+\frac{1}{q}=\frac{1}{r}$ and let $A \in S_p\left(\mathbb{X}, \mathbb{Y}\right), B \in S_q\left(\mathbb{X}, \mathbb{Y}\right)$. Then, $BA \in S_r\left(\mathbb{X}, \mathbb{Z}\right)$ and we have the bound: \[\Vert BA \Vert_{S_r\left(\mathbb{X}, \mathbb{Z}\right)} \leq 2^{1/r} \Vert A \Vert_{S_p\left(\mathbb{X}, \mathbb{Y}\right)} \Vert B \Vert_{S_q\left(\mathbb{Y}, \mathbb{Z}\right)}. \]
        \item\label{it:Schatten_composition_2} Let $A \in S_p\left(\mathbb{X}, \mathbb{Y}\right), B \in L\left(\mathbb{Y}, \mathbb{Z}\right), C \in L\left(\mathbb{Z}, \mathbb{X}\right)$. Then, $BA \in S_p\left(\mathbb{X}, \mathbb{Z}\right), AC \in S_p\left(\mathbb{Z}, \mathbb{Y}\right)$ and we have the bounds:
        \begin{align*}
            \Vert BA \Vert_{S_p\left(\mathbb{X}, \mathbb{Z}\right)} \leq \Vert A \Vert_{S_p\left(\mathbb{X}, \mathbb{Y}\right)} \Vert B \Vert, \qquad \Vert AC \Vert_{S_p\left(\mathbb{Z}, \mathbb{Y}\right)} \leq \Vert A \Vert_{S_p\left(\mathbb{X}, \mathbb{Y}\right)} \Vert C \Vert.
        \end{align*}
        \item Let $p, q>1$ with $\frac{1}{p}+\frac{1}{q}=1$. Then, $S_p\left(\mathbb{X}, \mathbb{Y}\right)^{\prime}=S_q\left(\mathbb{Y}^{\prime}, \mathbb{X}^{\prime}\right)$ with the dual pairing $\langle A, B \rangle=\tr(B^* A)$ for $A \in S_p\left(\mathbb{X}, \mathbb{Y}\right)$ and $B \in S_q\left(\mathbb{X}^{\prime}, \mathbb{Y}^{\prime}\right)$. 
    \end{enumerate}
\end{proposition}
\begin{proof}
    Part (i) follows from Theorem~1 of \cite{Gramsch1968}. Part (ii) and (iii) follow from Lemma~16.7 of \cite{Meise1992}. 

    Let $A \in S_p\left(\mathbb{X}, \mathbb{Y}\right)$ and $B \in S_q\left(\mathbb{X}^{\prime}, \mathbb{Y}^{\prime}\right)$. By part (ii) and the boundedness of the trace in $S_1$ (\cite[Lemma~16.23]{Meise1992}), we get: 
    $|\langle A, B \rangle|=|\tr(B^* A)| \leq \Vert B^* A \Vert_{S_1\left(\mathbb{Y}\right)} \leq \Vert B^* \Vert_{S_q\left(\mathbb{Y}, \mathbb{X} \right)}  \Vert A \Vert_{S_p\left(\mathbb{X}, \mathbb{Y}\right)}=\Vert B \Vert_{S_q\left(\mathbb{X}^{\prime}, \mathbb{Y}^{\prime} \right)}  \Vert A \Vert_{S_p\left(\mathbb{X}, \mathbb{Y}\right)}$. Hence, $S_q\left(\mathbb{Y}^{\prime}, \mathbb{X}^{\prime}\right) \subseteq S_p\left(\mathbb{X}, \mathbb{Y}\right)^{\prime}$ and $S_p\left(\mathbb{X}, \mathbb{Y}\right) \subseteq S_q\left(\mathbb{Y}^{\prime}, \mathbb{X}^{\prime}\right)^{\prime}$. By the Hahn-Banach theorem, we have the sequence
    \begin{align*}
        S_p(\mathbb{X}, \mathbb{Y}) \subseteq S_q\left(\mathbb{Y}^{\prime}, \mathbb{X}^{\prime}\right)^{\prime} \subseteq S_p(\mathbb{X}, \mathbb{Y})^{\prime \prime}.
    \end{align*}
    $S_p\left(\mathbb{X}, \mathbb{Y}\right)$ is a uniformly convex Banach space (\cite[Chapter~5]{Fack1986}) and therefore reflexive by Milman–Pettis theorem (\cite{Pettis1939}). Hence, $S_p\left(\mathbb{X}, \mathbb{Y}\right)^{\prime \prime} = S_p\left(\mathbb{X}, \mathbb{Y}\right)$ and the assertion follows.
\end{proof}

Using this proposition, we can prove that multiplication operators are Hilbert-Schmidt in suitable Sobolev spaces:

\begin{lemma}\label{lem:MS}
Let $\Omega\subset\mathbb{R}^d$ be a bounded Lipschitz domain, $S\in L^{\infty}(\Omega)$, and $s>d/4$, $s-1/2\notin \mathbb{N}_0$. (In particular, for $d\in\{2,3\}$ we may 
choose $s=1$.) 
Then 
$M_S\in \HS{H^s(\Omega)}{H^{-s}_0(\Omega)}$, and the following mapping 
is continuous
\begin{align}\label{eq:defiM}
\begin{aligned}
\mathcal{M}:&L^{\infty}(\Omega)\to \HS{H^s(\Omega)}{H^{-s}_0(\Omega)},\\
& S\mapsto M_S.
\end{aligned}
\end{align}
\end{lemma}
\begin{proof}
The condition $s-1/2\notin \mathbb{N}_0$ ensures that $H^s(\Omega)^{\prime} = H^{-s}_0(\Omega)$ (see, e.g., \cite[Chap.~4]{triebel:78}). 
    Let $\tilde{M_S}: L^2(\Omega) \to  L^2(\Omega), \tilde{M_S} \psi:=S \psi$. Then, we consider $M_S$ in the function spaces:
    \begin{align*}
H^s(\Omega) \stackrel{j}{\hookrightarrow} L^2(\Omega) \xrightarrow{\tilde{M_S}} L^2(\Omega) 
\stackrel{j^*}{\hookrightarrow} H_0^{-s}(\Omega).
\end{align*}
By Proposition~\ref{prop: p-Schatten-embeddings}, part~\ref{it:Schatten_embedding}, the embedding $j$ is an element of the Schatten class $S_p\left(H^s(\Omega), L^2(\Omega)\right)$ if $p>d/s$. Consequently, $j^*\in S_p\left(L^2(\Omega), H_0^{-s}(\Omega)\right)$. It follows from Proposition~\ref{prop: p-Schatten-embeddings}, parts~\ref{it:Schatten_composition} and \ref{it:Schatten_composition_2} 
that $M_S\in S_r(H^s(\Omega),H^{-s}_0(\Omega))$ if $\frac{1}{r}=\frac{1}{p}+\frac{1}{p}<\frac{2s}{d}$. As $\frac{2s}{d}>1/2$, $r=2$ is admissible. 
The continuity of $\mathcal{M}$ follows from the continuity of the mapping: $S \rightarrow \tilde{M_S}$.
\end{proof}

\begin{lemma}\label{lem:M_adjoint}
Under the assumptions of Lemma~\ref{lem:MS}, the adjoint operator 
$\mathcal{M}^*:\HS{H^{-s}_0(\Omega)}{H^s(\Omega)} \to L^{\infty}(\Omega)'$ takes 
values in the pre-dual $L^1(\Omega)\subset L^{\infty}(\Omega)^{\prime}$ 
of $L^{\infty}(\Omega)$ and 
\begin{align}\label{eq:Mstar_diag}
\mathcal{M}^* = \Diag.
\end{align}
\end{lemma}

\begin{proof}
Let $f \in L^{\infty}(\Omega)$, $j: H^s(\Omega) \hookrightarrow L^2(\Omega)$, and 
$\mathcal{A}\in \HS{H^{-s}_0(\Omega)}{H^s(\Omega)}$. 
It follows from Proposition~\ref{prop: p-Schatten-embeddings} that $\tilde{\mathcal{A}}:= j\mathcal{A}j^*\in S_1\left(L^2(\Omega)\right)$. We identify $\tilde{\mathcal{A}}$ and $\mathcal{A}$, i.e. a more precise formulation of \eqref{eq:Mstar_diag} is $\mathcal{M}^*(\mathcal{A}) = \Diag(\tilde{\mathcal{A}})$ for all $\mathcal{A}$. By the density result established at the beginning 
of the proof of Proposition~\ref{prop:diagonal}, it 
suffices to establish the relation for operators $\mathcal{A}$ with continuous kernel $A$. Choosing an orthonormal basis $\{e_k:k\in \mathbb{N}\}$ 
of $L^2(\Omega)$, we obtain
\begin{align*}
    \langle M_f,\mathcal{A}\rangle
   & = \tr\left(\mathcal{A}^*M_f\right)
    = \tr\left(j\mathcal{A}^*M_f j^{-1}\right)
    = \tr\left(\tilde{\mathcal{A}}^* \tilde{M}_f\right)
    = \sum_{k=1}^{\infty} \left\langle \tilde{\mathcal{A}}^*(fe_k),e_k\right\rangle\\
    & = \sum_{k=1}^{\infty} \int_{\Omega} \int_{\Omega}
    \overline{A(\bi{y},\bi{x})}f(\bi{y})e_k(\bi{y})\,\mathrm{d}\bi{y}\, \overline{e_k(\bi{x})}\,\mathrm{d}\bi{x} \\
    &= \int_{\Omega} \sum_{k=1}^{\infty}e_k(\bi{y}) \int_{\Omega}
    \overline{A(\bi{y},\bi{x})e_k(\bi{x})}\,\mathrm{d}\bi{x} 
    f(\bi{y})\mathrm{d}\bi{y}.
\end{align*}
Since $\Omega$ is bounded and hence 
$\overline{A(\bi{y},\cdot)}\in C(\overline{\Omega})
\subset L^2(\Omega)$, the completeness of $\{e_k\}$ implies that 
$\sum_{k=1}^{\infty}e_k(\bi{y}) \int_{\Omega}
    \overline{A(\bi{y},\bi{x})e_k(\bi{x})}\,\mathrm{d}\bi{x} =\overline{A(\bi{y},\bi{y})}$. This shows that 
 $\langle \mathcal{M}(f),\mathcal{A}\rangle = 
 \langle f,\Diag(\tilde{\mathcal{A}})\rangle$, completing the proof. 
\end{proof}

In Lemma \ref{lem:Frechet_G} we consider $M_{\partial v}$ and $M_{\partial \bi{A}}$ in the following function spaces:
\begin{align*}
    &M_{\partial v}: H^1(\Omega_0) \rightarrow H_0^{-1}(\Omega_0), \\
    &M_{\partial \bi{A}}: L^{2}(\Omega)^d \rightarrow H_0^{-1}(\Omega).
\end{align*}
The multiplication operator $M_{\partial v}$ has been discussed in Lemma~\ref{lem:MS} and Lemma~\ref{lem:M_adjoint}.
Although for $M_{\partial \bi{A}}$ we have less regularity, the following analogs still hold true:
\begin{lemma}
\label{lem:M_A}
    Let $d \in \{2, 3\}$ and $\partial \bi{A} \in W^{\infty}(\diver, \Omega_0)$. Then,
    \begin{enumerate}
        \item $M_{\partial \bi{A}} \in S_4\left(L^{2}(\Omega_0)^d, H_0^{-1}(\Omega_0)\right)$ and the following map is continuous:
    \begin{align*}
    &\tilde{\mathcal{M}}: L^{\infty}(\Omega_0, \Rd) \to S_4\left(L^{2}(\Omega_0)^d, H_0^{-1}(\Omega_0)\right)\\
    &\partial \bi{A} \mapsto M_{\partial \bi{A}}.
    \end{align*}
    \item 
    $\tilde{\mathcal{M}}^*:S_{4/3} (L^{2}(\Omega_0)^d, H^1(\Omega_0))= S_4\left(L^{2}(\Omega_0)^d, H_0^{-1}(\Omega_0)\right)^{\prime}\to L^{\infty}(\Omega_0)^{\prime} $ takes values in the pre-dual $L^1(\Omega_0) \subset L^{\infty}(\Omega_0)^{\prime}$. For an operator $\mathcal{B}
    = (\mathcal{B}_1,\dots,\mathcal{B}_d)\in S_{4/3} (L^{2}(\Omega_0)^d, H_1(\Omega_0))$ with continuous kernel $\bi{B}
    =(B_1,\dots, B_d):\Omega\times\Omega\to \C^d$ we have 
\[
\tilde{\mathcal{M}}^*\mathcal{B}=\biDiag\bi{B}, \qquad 
\biDiag\bi{B}:=(\Diag B_1,\dots,\Diag B_d).
\]
    \end{enumerate}
\end{lemma}

\begin{proof}
In this proof, $j$ will denote the embedding $H^1(\Omega_0) \hookrightarrow L^2(\Omega_0)$ and recall from Proposition~\ref{prop: p-Schatten-embeddings}, part \ref{it:Schatten_embedding} that $j\in S_4\left(H^1(\Omega_0), L^2(\Omega_0)\right)$ and hence $j^*\in S_4\left(L^2(\Omega_0),H^{-1}_0(\Omega_0)\right)$.

    \emph{Part (i):} We consider $M_{\partial \bi{A}} = j^*\circ \tilde{M}_{\partial \bi{A}}$ in the function spaces
    \begin{align*}
L^{2}(\Omega_0)^d \xrightarrow{\tilde{M}_{\partial \bi{A}}} L^2(\Omega_0) \stackrel{j^*}{\hookrightarrow} H_0^{-1}(\Omega_0),
\end{align*}
where $\tilde{M}_{\partial \bi{A}} \bi{\psi}=\partial \bi{A} \cdot \bi{\psi}$ for $\bi{\psi} \in L^{2}(\Omega_0)^d$.
The claim follows by Proposition~\ref{prop: p-Schatten-embeddings}.

\emph{Part (ii):}
Let $\tilde{\mathcal{B}}:=j\circ \mathcal{B}: L^{2}(\Omega_0)^d \to L^2(\Omega_0)$. 
Part (ii) of this proposition yields $\tilde{\mathcal{B}} \in S_1\left(L^{2}(\Omega_0)^d, L^2(\Omega_0)\right)$. As in Lemma \ref{lem:M_adjoint}, the assertion now follows.
\end{proof}

\section{Fr\'echet derivative and adjoint of the forward operator}
\label{sec: frechet derivative}

A characterization of the adjoint of $S\mapsto\Tr_{\Gamma} \mathcal{G} M_S \mathcal{G}^*\Tr_{\Gamma}^*$ was given in \cite{Hohage2020}. There, a characterization of the adjoint of $S \mapsto M_S$ in a functional analytic framework was circumvented, resulting in a rather technical formulation of the result.

With the results of the previous section, the proof of the following central results is 
now mostly straightforward. 
\begin{theorem}
\label{the: HS-operator}
Assumptions~\ref{ass:well_posed} 
and \ref{ass:solvability_noise} hold true for some wave number $k\in\C$ and 
$d\in\{2,3\}$. 
Let $\mathbb{X}:=\mathbb{X}_{\mathcal{G}} \times 
L^{\infty}(\Omega_0,\R)$ with $\mathbb{X}_{\mathcal{G}}$ defined in 
\eqref{eq:defi_XG}
and let $\Bfrak:=\Bfrak_k\times L^{\infty}(\Omega_0;[0,\infty))$. 
Then the following holds true: 
\begin{enumerate}
\item The forward operator~\eqref{eq:extended_forward_op} is well-defined 
and continuous as a mapping 
\[
\mathcal{C}:\Bfrak \to \HSs{L^2(\Gamma)}, 
\]
and it is Fr\'echet differentiable on the interior of $\Bfrak$. The derivative
$\mathcal{C}^{\prime}[v,\bi{A},S]:\mathbb{X}\to \HSs{L^2(\Gamma)}$ is given 
by 
\begin{align*}
&\mathcal{C}^{\prime}[v,\bi{A},S](\partial v,\partial\bi{A},\partial S)\\
&\hspace{1cm}= 2\Re\left(\Tr_{\Gamma}\mathcal{G}_{v,\bi{A}}
\left(-M_{\partial v}+2iM_{\partial \bi{A}}\cdot\nabla \right)\mathcal{G}_{v,\bi{A}}\left(M_{S}+\Tr^*_{\partial \Omega} \bCov[s_{\partial\Omega}] \Tr_{\partial \Omega} \right)
\mathcal{G}_{v,\bi{A}}^*\Tr_{\Gamma}^*\right)\\
&\hspace{1cm}+ \Tr_{\Gamma} \mathcal{G}_{v,\bi{A}}M_{\partial S}\mathcal{G}_{v,\bi{A}}^*\Tr_{\Gamma}^*
\end{align*}
where $\Re(\mathcal{A}):=\frac{1}{2}(\mathcal{A}+\mathcal{A}^*)$. 
\item The adjoint $\mathcal{C}^{\prime}[v,\bi{A},S]^*:\HSs{L^2(\Gamma)}
\to \mathbb{X}^{\prime}$ of $\mathcal{C}^{\prime}[v,\bi{A},S]$ takes values in the pre-dual 
$L^{1}(\Omega_0;\C)\times L^{1}(\Omega_0;\Rd) \times 
L^{1}(\Omega_0;\R)\subset \mathbb{X}^{\prime}$ of $\mathbb{X}$ and is given by 
\begin{align*}
&\mathcal{C}^{\prime}[v,\bi{A},S]^*\mathcal{D}
= \left(\begin{array}{c}
-2\Diag\left(\mathcal{E} (M_S+\Tr_{\partial \Omega}^* \bCov[s_{\partial\Omega}] \Tr_{\partial \Omega}) \mathcal{G}_{v,\bi{A}}^* \right)\\
-4i\biDiag\left(\mathcal{E} (M_S+\Tr_{\partial \Omega}^* \bCov[s_{\partial\Omega}] \Tr_{\partial \Omega}) (\nabla\mathcal{G}_{v,\bi{A}})^*\right)\\
\Diag\mathcal{E}\end{array}\right),\\
&\mathcal{E}:=\mathcal{G}_{v,\bi{A}}^*\Tr_{\Gamma}^*\Re(\mathcal{D})\Tr_{\Gamma}\mathcal{G}_{v,\bi{A}}.
\end{align*}
\end{enumerate}
\end{theorem}

\begin{proof}
\emph{Part (i):} 
Let $\mathcal{C}_1[v,\bi{A},S]:= \Tr_{\Gamma} \mathcal{G}_{v,\bi{A}} M_{S} \mathcal{G}_{v,\bi{A}}^*\Tr_{\Gamma}^*.$ and $\mathcal{C}_2[v,\bi{A},S]:= \Tr_{\Gamma} \mathcal{G}_{v,\bi{A}} \Tr^*_{\partial \Omega} \bCov[s_{\partial\Omega}] \Tr_{\partial \Omega} \mathcal{G}_{v,\bi{A}}^*\Tr_{\Gamma}^*.$
We consider the factors defining $\mathcal{C}_1[v,\bi{A},S]$ in the 
following function spaces:
\begin{align*}
L^2(\Gamma)\hookrightarrow H^{-1/2}(\Gamma)
\xrightarrow{\Tr_{\Gamma}^*} H_0^{-1}(\Omega) \xrightarrow{\mathcal{G}^*}  H^1(\Omega) \xrightarrow{M_S} H_0^{-1}(\Omega) \xrightarrow{\mathcal{G}} H^1(\Omega) \xrightarrow{\Tr_{\Gamma}}H^{1/2}(\Gamma)\hookrightarrow L^2(\Gamma).
\end{align*}
Here $M_S: H^1(\Omega) \rightarrow H_0^{-1}(\Omega)$ is Hilbert-Schmidt by 
Lemma~\ref{lem:MS}, and all other operators are bounded. By part (iii) of Proposition \ref{prop: p-Schatten-embeddings}, it follows that $\mathcal{C}_1[v,\bi{A},S]$ is Hilbert-Schmidt. 

Similarly, we consider the factors defining $\mathcal{C}_2[v,\bi{A},S]$ in the following function spaces:
\begin{align*}
&\fl L^2(\Gamma)\hookrightarrow H^{-1/2}(\Gamma)
\xrightarrow{K^*} H^{1/2}(\partial \Omega) \hookrightarrow L^2(\partial \Omega) \xrightarrow{\bCov[s_{\partial\Omega}]} L^2(\partial \Omega)\\
&\hookrightarrow H^{-1/2}(\partial \Omega)  \xrightarrow{K} H^{1/2}(\Gamma) \hookrightarrow L^2(\Gamma).
\end{align*}
By part (i) of Proposition~\ref{prop: p-Schatten-embeddings}, every embedding is an element of $S_8$. By part \ref{it:Schatten_composition}, \ref{it:Schatten_composition_2} of Proposition~\ref{prop: p-Schatten-embeddings}, it follows that $\mathcal{C}_2[v,\bi{A},S]$ is Hilbert-Schmidt. Hence, $\mathcal{C}[v,\bi{A},S]$ is Hilbert-Schmidt.
Together with Lemma~\ref{lem:Frechet_G}, it follows that $\mathcal{C}$ is 
continuous. Fr\'echet differentiability and the formula for the derivative 
follow from Lemma~\ref{lem:Frechet_G} and the chain rule. 

\emph{Part (ii):} If $\mathbb{X}_1,\dots,\mathbb{X}_4$ are Hilbert spaces, 
$\mathcal{A}\in L\left(\mathbb{X}_1,\mathbb{X}_2\right)$ and 
$\mathcal{B}\in L\left(\mathbb{X}_3,\mathbb{X}_4\right)$, then a straightforward computation 
shows that the adjoint of the linear mapping 
$\HS{\mathbb{X}_2}{\mathbb{X}_3}\to \HS{\mathbb{X}_1}{\mathbb{X}_4}$, 
$\mathcal{T}\mapsto \mathcal{BTA}$ is given by the mapping 
$\HS{\mathbb{X}_1}{\mathbb{X}_4}\to \HS{\mathbb{X}_2}{\mathbb{X}_3}$, $\mathcal{S}\mapsto \mathcal{B}^*
\mathcal{S}\mathcal{A}^*$ and that 
$\Re \in L\left(\HSs{\mathbb{X}_j}\right)$ is a self-adjoint projection operator.
Note that $\bCov[s_{\partial\Omega}] \in S_4\left(H^{1/2}(\partial \Omega), H^{-1/2}(\partial \Omega)\right)$ from Proposition~\ref{prop: p-Schatten-embeddings}. 
Furthermore, by Proposition~\ref{prop: p-Schatten-embeddings} and Lemma~\ref{lem:MS}, $\mathcal{E} \in HS\left(H_0^{-1}(\Omega_0), H^1(\Omega_0)\right), M_S \in HS\left(H^1(\Omega_0), H_0^{-1}(\Omega_0)\right)$. Hence, $\mathcal{E} (M_S+\Tr_{\partial \Omega}^* \bCov[s_{\partial\Omega}] \Tr_{\partial \Omega}) \in S_{4/3}\left(H_0^{-1}(\Omega_0)\right)$.
Now, the assertion follows from Lemma~\ref{lem:M_adjoint} and part (ii) of Lemma~\ref{lem:M_A}. 
\end{proof}

Introducing so-called forward propagators $H_{\alpha}$ and backward propagators $H_{\beta}$ by 
\begin{equation}
\begin{split}
    &H^{v, \bi{A}}_{\alpha_v}
    :=H^{v, \bi{A}}_{\alpha_{\bi{A}}}
    :=H^{v, \bi{A}}_{\alpha_S}
    :=H_{\beta_S}
    :=\Tr_{\Gamma} \mathcal{G}_{v,\bi{A}} 
    \in L\left(H_0^{-1}(\Omega), L^2(\Gamma)\right), \\
    &H^{v, \bi{A}}_{\beta_v}
    :=\Tr_{\Gamma} \mathcal{G} M_S \mathcal{G}_{v,\bi{A}}^* \in L\left(H_0^{-1}(\Omega), L^2(\Gamma)\right), \\
    &H^{v, \bi{A}}_{\beta_{\bi{A}}}:=\Tr_{\Gamma} \mathcal{G}_{v,\bi{A}} M_S (\nabla \mathcal{G}_{v,\bi{A}})^*   \in L\left(H_0^{-1}(\Omega), L^2(\Gamma)\right),
\end{split} 
\end{equation}
the Fr\'echet derivative and its adjoint in Theorem~\ref{the: HS-operator} can be reformulated as 
\begin{align}
&\mathcal{C}^{\prime}[v,\bi{A},S](\partial v,\partial\bi{A},\partial S)=-2\Re(H^{v, \bi{A}}_{\alpha_v} M_{\partial_v} {H^{v, \bi{A}}_{\beta_v}}^*)+2 \Re(H^{v, \bi{A}}_{\alpha_{\bi{A}}} M_{2i \partial \bi{A}} {H^{v, \bi{A}}_{\beta_{\bi{A}}}}^*)+H^{v, \bi{A}}_{\alpha_S} M_{\partial_S} {H^{v, \bi{A}}_{\beta_S}}^*, \nonumber\\
& \mathcal{C}^{\prime}[v,\bi{A},S]^*\mathcal{D}
= \left(\begin{array}{c}
-2\Diag\left({H^{v, \bi{A}}_{\alpha_v}}^* \Re(\mathcal{D}) H^{v, \bi{A}}_{\beta_v}\right)\\
-4i\,{\biDiag}\left({H^{v, \bi{A}}_{\alpha_{\bi{A}}}}^* \Re(\mathcal{D}) H^{v, \bi{A}}_{\beta_{\bi{A}}}\right)\\
\Diag\left({H^{v, \bi{A}}_{\alpha_S}}^* \Re(\mathcal{D}) H^{v, \bi{A}}_{\beta_S}\right)\end{array}\right).
\label{eq: Frechet-derivative_propagator}
\end{align}

These propagators $\mathcal{H}_{\alpha}, \mathcal{H}_{\beta}$ have a physical interpretation in helioseismology that will be discussed in Section~\ref{sec: iterative helioseismic holography}.

\section{On the algorithmic realization of iterative regularization methods}\label{sec: Algorithm}

For notational simplicity, we will use $\bi{q}=(v, \bi{A}, S)$ throughout this section. Then we formally have to solve the operator equation 
\[
\mathcal{C}[\bi{q}]=\Corr\qquad \text{with }
\Corr:=\frac{1}{N}\sum_{n=1}^N \Tr_{\Gamma}\psi_n \otimes 
\overline{\Tr_{\Gamma}\psi_n}.
\]

\subsection{Avoiding the computation of $\Corr$}
\label{sec: Avoid Corr}
Computing $\Corr$ from the primary data $\Tr_{\Gamma} \psi_n$ in 
a preprocessing step drastically increases the dimensionality of the 
data. In helioseismology, the data set with the best resolution consists of Doppler images of size $4096 \times 4096$. This leads to approximately $10^{14}$ independent two-point correlations, at each frequency. Hence, the surface cross-correlation is a noisy five-dimensional data set of immense size, which is infeasible to use in inversions directly. 
Moreover, these two-point correlations are extremely noisy. 
In traditional approaches such as time-distance helioseismology, one usually reduces the cross-correlation in an a priori step to a smaller number of physically interpretable quantities with an acceptable signal-to-noise ratio. However, this leads to a significant loss of information, see \cite{Pourabdian2020}. 

To use the full information content of $\Corr$ without the need to ever compute these correlations, we exploit the fact that the adjoint of the Fr\'echet derivative of the forward operator is of the form $\mathcal{C}^{\prime}[\bi{q}]^*\mathcal{D} = \Diag({\mathcal{H}_{\alpha}^q}^* \Re(\mathcal{D})\mathcal{H}_{\beta}^q)$, see 
eq.~\eqref{eq: Frechet-derivative_propagator}.
As $\Re(\Corr)=\Corr$, pulling the sum outside yields
\begin{align}\label{eq:backpropagation_correlation}
\mathcal{C}^{\prime}[\bi{q}]^*\Corr = \frac{1}{N}\sum_{n=1}^N
\Diag\left({\mathcal{H}_{\alpha}^q}^*\Tr_{\Gamma}\psi_n \otimes {\mathcal{H}_{\beta}^q}^*\Tr_{\Gamma}\psi_n\right).
\end{align}
We will show in Section~\ref{sec: iterative helioseismic holography} that traditional helioseismic holography can be interpreted as the application of $\mathcal{C}^{\prime}[\bi{q}]^*$ to $\Corr$. Since $ \frac{1}{N}\sum_{n=1}^N\Diag(\dots)$ can be interpreted as computing diagonal correlations of the back-propagated signals, eq.~\eqref{eq:backpropagation_correlation} may be paraphrased as first 
back-propagating signals and then correlating them, rather than vice versa.

\subsection{Iterative regularization methods without image space vectors}
For ill-posed inverse problems, the adjoint of the linearized forward operator is typically a bad approximation of the inverse. To obtain a quantitative imaging method, we can improve the approximation in \eqref{eq:backpropagation_correlation} by implementing an iterative regularization method. We will focus on the Iteratively regularized Gauss-Newton Method (IRGNM) with inner Conjugate Gradient iterations, but the discussion below also applies to other commonly used methods such as Landweber iteration or the Newton-CG method. IRGNM is defined by 
\begin{equation}
\label{eq: Iteration}
\begin{split}
 \delta \bi{q}_n   &=\operatorname{argmin}_{\bi{q}}
 \left\|\mathcal{C}[\bi{q}_{n}]+\mathcal{C}^{\prime}[\bi{q}_n]\bi{q}-\Corr\right\|_{\mathbb{Y}}^2 + \alpha_n\|\bi{q}+\bi{q}_n-\bi{q}_0\|_{\mathbb{X}}^2\\
&= \left( \mathcal{C}^{\prime}[\bi{q}_n]^* \mathcal{C}^{\prime}[\bi{q}_n] +\alpha_n \Id \right)^{-1} \left(\mathcal{C}^{\prime}[\bi{q}_n]^* \left( \Corr-\mathcal{C}[\bi{q}_n] \right)+\alpha_n (\bi{q}_0-\bi{q}_n )\right)\\
    \bi{q}_{n+1}&=\bi{q}_n+\delta \bi{q}_n.
\end{split}
\end{equation}
Here $\bi{q}_0$ defines the initial guess. Since the image space $\mathbb{Y}$ of the forward operator is high dimensional, direct evaluations of $\mathcal{C}[\bi{q}]$ and $\mathcal{C}^{\prime}[\bi{q}]$ must be avoided. However, IRGNM with inner CG iterations as well as other iterative regularization methods only require the operations
$\bi{q}\mapsto \mathcal{C}^{\prime}[\bi{q}]^* \mathcal{C}[\bi{q}]
$ and
\begin{equation}\label{eq:defi_sensitivity_kernel}
   \left( \mathcal{C}^{\prime} [\bi{q}]^*\ \mathcal{C}^{\prime}[\bi{q}]\ \partial \bi{q}\right) (\bi{x})= \int_{\Omega} K(\bi{x}, \bi{y}) \partial \bi{q}(\bi{y}) \,\mathrm{d} \bi{y}.
\end{equation}
We will refer to the integral kernels $K$ of $\mathcal{C}^{\prime} [\bi{q}]^*\ \mathcal{C}^{\prime}[\bi{q}]$ as sensitivity kernels for the normal equation. 
In Section~\ref{sec: sensitivity_Kernel} they will be described for various settings of interest in terms of forward-backward operators
\begin{align}\label{eq:forwardbackward}
    \mathcal{F}_{\alpha, \beta}:=\mathcal{H}_{\alpha}^* \mathcal{H}_{\beta} : H_0^{-1}(\Omega) \rightarrow H^1(\Omega), \qquad& (\mathcal{F}_{\alpha, \beta} \psi)(\bi{x})=\int_{\Omega} F_{\alpha, \beta}(\bi{x}, \bi{y}) \psi(\bi{y}) \,\mathrm{d} \bi{y}.
\end{align}
In our numerical tests in helioseismology reported in Section~\ref{sec: Inversions}, the bottleneck concerning computation time is the evaluation of the Green function involved in the definitions of the propagators $\mathcal{H}_{\alpha}$ and $\mathcal{H}_{\beta}$.
To accelerate these computations, we use separable reference Green's functions $G_0:= G_{\bi{q}_0}$ discussed in \ref{sec: Green} and \ref{sec: Uniform Green} and the corresponding integral operator $\mathcal{G}_0$ as well as the algebraic identity
\begin{subequations}\label{Green_approx}
\begin{align}
\mathcal{G}_{\bi{q}}=\left[ \Id+\mathcal{G}_0 (L_{\bi{q}}-L_0) \right]^{-1} \mathcal{G}_0. 
\end{align}
This identity, with $L_{\bi{q}}:=\mathcal{G}_{\bi{q}}^{-1}$ and $L_0:=\mathcal{G}_0^{-1}$, is equivalent to 
$     \mathcal{G}_0-\mathcal{G}_{\bi{q}} = \mathcal{G}_0 (L_{\bi{q}}-L_0) \mathcal{G}_{\bi{q}}.$  
The operators $L_{\bi{q}}, L_0:H^1(\Omega)\to H^{-1}_0(\Omega)$ represent 
the corresponding sesquilinear forms in eq.~\eqref{eq:weak_formulation} in the 
proof of Proposition \ref{prop:wellposed} and involve both the differential operator and the boundary condition. As both the boundary operator $B$ and the leading order differential operator are independent of the parameters $\bi{q}$, they cancel out, and 
\begin{align}
(L_{\bi{q}}-L_0)\psi = (v-v_0)\psi - 2i (\bi{A}-\bi{A}_0)\cdot \nabla \psi.
\end{align}
\end{subequations}

This approach is efficient since the operator $\mathcal{G}_0$ can be solved with one-dimensional code and the operator the calculation of $\left[ \Id+\mathcal{G}_0 (L_{\bi{q}}-L_0) \right]$ can typically be restricted to a supported area of $L_{\bi{q}}-L_0$.
Usually, we compute a pivoted LU-decomposition of $\Id+\mathcal{G}_0 (L_{\bi{q}}-L_0) $ and solve for a list of input sources $G(\cdot, \bi{x}), \bi{x} \in \Gamma$. Furthermore, we can use low-rank approximations for $\mathcal{G}_0$ 
based on the expansions in \ref{sec: Green} for solar-like medium and \ref{sec: Uniform Green} for uniform medium.  

\subsection{Noise and likelihood modelling}
\label{sec: Noise model}
In this section, we study the noise model in order to step forward to the full likelihood modeling and to create realistic noise. The main noise term is realization noise. Recall that the wavefield $\psi$ is a realization of a Gaussian random process with covariance operator $\mathcal{C}[\bi{q}]$. 

The covariance matrix of Gaussian correlation data can be computed by Isserlis theorem \cite{Isserlis1918} and is given by fourth-order correlations (e.g. \cite{Gizon2004}, \cite{Fournier2014}, \cite{Gizon2018}):

\begin{eqnarray*}
    \fl \E{\psi(\bi{r}_1) \overline{\psi(\bi{r}_2)} \overline{\psi(\bi{r}_3)} \psi(\bi{r}_4)}=\E{\psi(\bi{r}_1) \overline{\psi(\bi{r}_2)}} \E{\overline{\psi(\bi{r}_3)} \psi(\bi{r}_4)}\\
    +\E{\psi(\bi{r}_1) \overline{\psi(\bi{r}_3)}} \E{\overline{\psi(\bi{r}_2)} \psi(\bi{r}_4)}+\E{\psi(\bi{r}_1) \psi(\bi{r}_4)} \E{\overline{\psi(\bi{r}_2)} \psi(\bi{r}_3)}.
\end{eqnarray*}

The third term vanishes as $\E{\psi(\bi{r}_1) \psi(\bi{r}_2)}=\int_{\Omega} \int_{\Omega} G(\pmb{r}_1, \pmb{z}_1) G(\pmb{r}_2, \pmb{z}_2)\E{s(\bi{z}_1) s(\bi{z}_2)} \,\mathrm{d} \bi{z}_1 \mathrm{d} \bi{z}_2$ and $\E{s(\bi{z}_1) s(\bi{z}_2)} = 0$. Hence, we observe

\begin{eqnarray}
    \nonumber
    \fl \Cov(C(\bi{r}_1, \bi{r}_2), C(\bi{r}_3, \bi{r}_4))&=\E{\psi(\bi{r}_1) \overline{\psi(\bi{r}_2)} \overline{\psi(\bi{r}_3)} \psi(\bi{r}_4)}-\E{\psi(\bi{r}_1) \overline{\psi(\bi{r}_2)}} \E{\overline{\psi(\bi{r}_3)} \psi(\bi{r}_4)}\\
    \nonumber
    \fl&=\E{\psi(\bi{r}_1) \overline{\psi(\bi{r}_3)}} \E{\overline{\psi(\bi{r}_2)} \psi(\bi{r}_4)}=\mathcal{C}(\bi{r}_1, \bi{r}_3) \mathcal{C} (\bi{r}_4, \bi{r}_2).
\end{eqnarray}

Therefore, we can define the data covariance operator by 
\begin{equation}
\label{eq: C_4}
    \mathcal{C}_4[\bi{q}] \in L \left( L^2(\Gamma) \times L^2(\Gamma) \right), \qquad \mathcal{C}_4[\bi{q}] (f \otimes g)=\mathcal{C}[\bi{q}](f) \otimes \mathcal{C}[\bi{q}](g).
\end{equation}
Hence, if we choose a quadratic log-likelihood approximation, we are formally lead to replace $\|\cdot\|_{\mathbb{Y}}^2$ in \eqref{eq: Iteration} by $\|\mathcal{C}_4[\bi{q}_n]^{-1/2}\cdot\|_{\mathbb{Y}}^2$.
However, with this replacement, the iteration \eqref{eq: Iteration} would in general not be well defined and numerically unstable since $\mathcal{C}_4[\bi{q}_n]$ is not boundedly invertible. 
Note that the operator $\mathcal{C}_4[\bi{q}_n]$ is not boundedly invertible. Even if $\mathcal{C}_4[\bi{q}_n]$ is injective, the inverse is given by $\mathcal{C}_4[\bi{q}_n]^{-1} = \mathcal{C}[\bi{q}_n]^{-1} \otimes \mathcal{C}[\bi{q}_n]^{-1}$, and this cannot be bounded in infinite dimensions since $\mathcal{C}[\bi{q}_n]$ is Hilbert-Schmidt. Hence, the mentioned replacement is not well defined and numerically unstable.
Therefore, we choose a bounded, self-adjoint, positive semi-definite approximation 
\begin{align}\label{eq:defi_Gamma_n}
 \mathcal{C}[\bi{q}_n]^{-1}  \approx \Gamma_n \in L(L^2(\Gamma)),
\end{align}
e.g., by a truncated eigenvalue decomposition or by 
Lavrentiev regularization $\Gamma_n= (\beta \Id+\mathcal{C}[\bi{q}_n])^{-1}$. 
Then $\mathcal{C}_4[\bi{q}_n]^{-1}\approx \Gamma_n\otimes \Gamma_n$ and 
$\mathcal{C}_4[\bi{q}_n]^{-1/2}\approx \Gamma_n^{1/2}\otimes \Gamma_n^{1/2}$. The parameter $\beta$ may model the presence of 
measurement and modelling errors in addition to the realization noise 
modelled by $\mathcal{C}_4[\bi{q}_n]^{-1/2}$. Then the iteration 
\eqref{eq: Iteration} is replaced by
\begin{equation*}
\begin{split}
 \delta \bi{q}_n   &=\operatorname{argmin}_{\bi{q}}
 \left\|(\Gamma_n^{1/2}\otimes \Gamma_n^{1/2})
 \mathcal{C}[\bi{q}]+\mathcal{C}^{\prime}[\bi{q}_n]\bi{q}-\Corr\right\|_{\mathbb{Y}}^2 + \alpha_n\|\bi{q}+\bi{q}_n-\bi{q}_0\|_{\mathbb{X}}^2\\
&= \left( \mathcal{C}^{\prime}[\bi{q}_n]^* 
(\Gamma_n\otimes \Gamma_n)
\mathcal{C}^{\prime}[\bi{q}_n] +\alpha_n I \right)^{-1} \left(\mathcal{C}^{\prime}[\bi{q}_n]^* (\Gamma_n\otimes \Gamma_n)
\left( \Corr-\mathcal{C}[\bi{q}_n] \right)+\alpha_n (\bi{q}_0-\bi{q}_n )\right)\\
    \bi{q}_{n+1}&=\bi{q}_n+\delta \bi{q}_n.
\end{split}
\end{equation*}
Note that for numerical efficiency it is very fortunate that the 
covariance operator $\mathcal{C}_4$ of the correlation data has 
the separable structure \eqref{eq: C_4}. Further note from the 
second line of the last equation that we only need $\Gamma_n$, not 
$\Gamma_n^{1/2}$. 
%

\section{Forward problems in local helioseismology}
\label{sec: Helmholtz}

In this section, we discuss applications of the model problem 
considered in the previous sections to helioseismology. 
\subsection{Acoustic oscillations in the Sun}
\label{sec: assumptions}

$\Omega_0$ will denote the interior of the Sun (typically $\Omega_0=B(0, R_{\odot})$ with $R_{\odot}=696\,$Mm), whereas 
$\Omega$ may also include parts of the solar atmosphere. 
The measurement region we consider is an open subset $\Gamma$ of the visible surface $\partial \Omega_0$, accounting for the fact that in typical helioseismic applications, measurements are only available on the near side of the solar surface. Given that solar oscillations near the solar surface are primarily oriented in the radial direction \cite{CD2003}, there is also a lack of Doppler information near the poles. This phenomenon results in leakage, causing challenges such as incomplete decoupling of normal modes of oscillation (e.g. \cite{Schou1994b, Hill1998}). In the subsequent analysis, we will exclusively work in the frequency domain.

The propagation of acoustic waves in a heterogeneous medium like the Sun can be described by the differential equation
\begin{eqnarray}
    \label{eq: vectorial}
    -(\omega+i \gamma+i \bi{u} \cdot \nabla)^2 \bi{\zeta}-\frac{1}{\rho} \nabla(\rho c^2\nabla \cdot \bi{\zeta})=\bi{F},
\end{eqnarray}
where we have ignored gravitational effects and have assumed an adiabatic approximation (\cite{Gizon2017}). The random source term $\bi{F}$ describes the stochastic excitation of waves by turbulent motions and $\bi{\zeta}$ is the Lagrangian wave displacement vector. As usual, we denote with $\rho$ the density, $c$ the sound speed, $\gamma$ the damping, and $\bi{u}$ the flow field.
If we furthermore neglect second order terms in $\gamma, \bi{u}$, equation \eref{eq: vectorial} can be converted into a Helmholtz-like equation (\cite{Gizon2018}, inspired by \cite{Lamb1908})
\begin{eqnarray}
    \label{eqn:Helmholtz}
    L \psi:=-(\Delta + V) \psi -\frac{2 i \omega}{\rho^{1/2} c} \rho \bi{u} \cdot \nabla \frac{\psi}{\rho^{1/2} c}=s,
\end{eqnarray}
where $\psi=\rho^{1/2} c^2 \nabla \cdot \bi{\zeta}$ is the scaled wavefield and $s=\rho^{1/2} c^2 \nabla \cdot \bi{F}$ a stochastic source term. The potential $V$ is defined by
\begin{eqnarray}
\label{eqn:wavenumber}
V=\frac{\omega^2+ 2 i \omega \gamma-\omega_c^2}{c^2},\qquad& \omega_c^2=\rho^{1/2}c^2 \Delta (\rho^{-1/2}).
\end{eqnarray}
The frequency $\omega_c$ is recognized as the acoustic cutoff frequency. This cutoff frequency arises due to the abrupt decline in density near the solar surface and results in the trapping of acoustic modes with frequencies below the acoustic cutoff frequency. Modes with frequencies surpassing the acoustic cutoff frequency can propagate through the solar atmosphere. 

The conditions on $c, \rho, \gamma, \bi{u}$ are summarized in Assumption~\ref{ass: helioseismic setting}.
\begin{assumption}
\label{ass: helioseismic setting}
    Suppose that for some $B_0\in L\left(H^{1/2}(\partial\Omega),H^{-1/2}(\partial\Omega)\right)$, $k\in\C$ and some set $\mathcal{A}_{c_{\text{min}}, \rho_{\text{min}}, k} \subset W^{1, \infty}(\Omega, \R) \times W^{2, \infty}(\Omega, \R) \times L^{\infty}(\Omega, [0, \infty)) \times  W^{\infty}(\diver,\Omega) \times L^{\infty}(\Omega, [0, \infty))$ of admissible parameters $c, \rho, \gamma, \bi{u}, S$ containing 
    some reference parameters $c_{\mathrm{ref}}, \rho_{\mathrm{ref}}, \gamma_{\mathrm{ref}}, \bi{u}_{\mathrm{ref}}, S_{\mathrm{ref}}$
    such that the following holds true:
\begin{subequations}\label{eqs:helioassumptions}
\begin{align}
\label{eq: parameter_regularity}
&\inf_{\bi{x} \in \Omega} c \geq c_{\text{min}} >0, \quad \inf_{\bi{x} \in \Omega} \rho \geq \rho_{\text{min}} >0, \\
\label{eq: parameter_exterior}
&q = q_{\mathrm{ref}}\qquad \mbox{for }q\in\{c, \rho, \gamma, \bi{u}, S\} 
&&\mbox{in }\Omega\setminus \Omega_0,\\
\label{eq:nomassflow}
&\bi u =0, \, S=0&&\text{in }\Omega\setminus \Omega_0,\\
\label{ass: conservation of mass}
&\diver(\rho \bi{u})=0&&\text{on }\Omega,\\
\label{eq: BC}
&B: H^{1/2}(\partial \Omega) \to H^{-1/2}(\partial \Omega) \text{ satisfies the conditions~\eqref{eq:assImB}--\eqref{eq:assBcompact} }
\end{align}
\end{subequations}
\end{assumption}

For the flow field, we incorporate a mass conservation constraint (equation~\eqref{ass: conservation of mass}). Additionally, we assume that the flow field does not intersect the computational boundary (equation~\eqref{eq:nomassflow}). 
 Various boundary conditions, in particular radiation boundary conditions and learned infinite elements, and their efficacy are extensively discussed in \cite{Barucq2018,Fournier2017,Preuss2020}. It is notable that the most popular choices of boundary conditions in helioseismology, such as radiation boundary conditions, Sommerfeld boundary conditions, or free boundary conditions, are incorporated in Assumption~\eqref{eq: BC}.

We define the operator $\mathcal{P}$ that transforms the parameters in the wave equation~\eqref{eqn:Helmholtz} into the form of equation~\eqref{eqs:forward} by
\begin{align}
\label{eq: recast}
\mathcal{P}: \mathcal{A}_{c_{\text{min}}, \rho_{\text{min}}, k} \rightarrow \mathcal{B}_k, \qquad \mathcal{P}(c, \rho, \gamma, \bi{u}, S)=(v, \bi{A}, S), 
\end{align}
where 
\begin{subequations}
\begin{align}
    \label{eq: k, A}
    & k^2=\frac{\omega^2+2i\omega \gamma}{c_0^2}-\frac{1}{4 H^2}, \hspace{2cm} \bi{A}=\omega \frac{1}{c^2} \bi{u}\\ 
    \label{eq: v}
    &v=k^2-\frac{\omega^2+2i \gamma \omega}{c^2}+\rho^{1/2} \Delta (\rho^{-1/2})-2 i \omega \frac{1}{\rho^{1/2} c} \rho \bi{u} \cdot \nabla \frac{1}{\rho^{1/2} c}.
\end{align}
\end{subequations}

\begin{figure}
   \centering
   \includegraphics[width=\hsize]{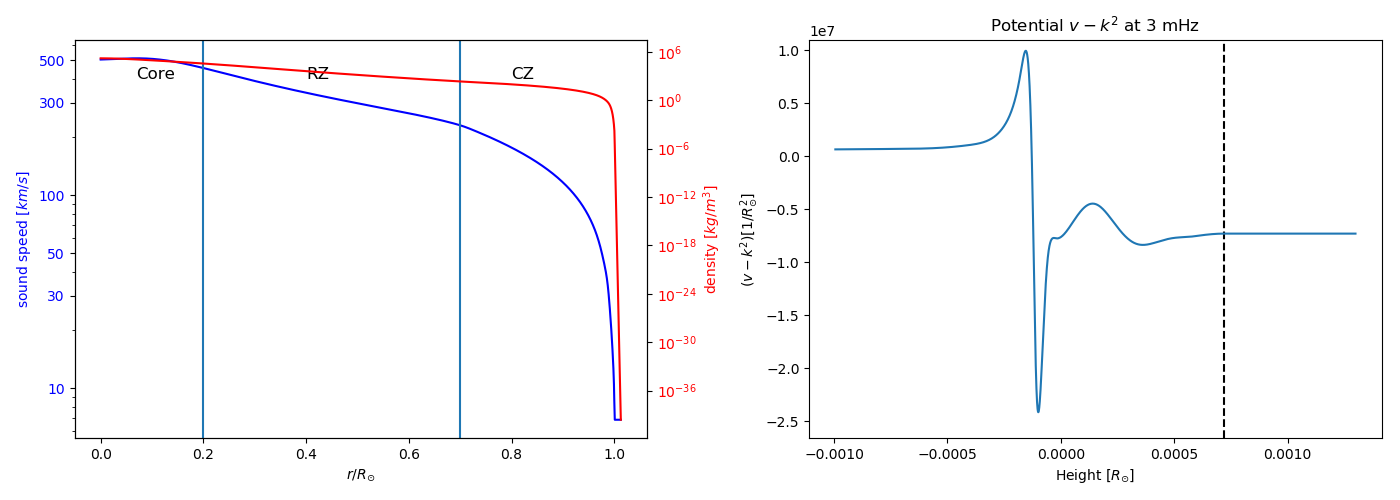}
      \caption{The left panel shows the sound speed and density obtained from the Solar Model S \cite{Dalsgaard1996} in the solar core, the convection zone (CZ), and the radiation zone (RZ). The right panel shows the potential close to the surface for $\omega/2 \pi=3~$mHz.}
         \label{fig: solar model}
   \end{figure}

In Figure~\ref{fig: solar model}, we present the acoustic sound speed, the density, and the scalar potential $v$ as obtained from the Solar Model S and smoothly extended to the atmosphere. Modeling the forward problem for the Sun remains challenging due to the substantial density gradients near the surface, leading to strong variations of the  scalar potential $v$ near the solar surface.

\begin{lemma}
\label{lem: helioseismic setting}
The operator $\mathcal{P}$, defined in equation~\eqref{eq: recast}, is well-defined in the sense that for all parameters $(c, \rho, \gamma, \bi{u}, S) \in \mathcal{A}_{c_{\text{min}}, \rho_{\text{min}}, k}$
we have $\mathcal{P}(c, \rho, \gamma, \bi{u}, S) \in \mathcal{B}_k$, and this map is continuous.
\end{lemma}

\begin{proof}
By equations~\eqref{eq: k, A}, \eqref{eq: v}, we have $v \in L^{\infty} (\Omega, \C)$,  $\bi{A} \in W^{\infty}(\diver, \Omega_0)$, and the mapping is continuous. 
The conditions~\eqref{eq:assImB}--\eqref{eq:assBcompact} are obviously satisfied, and \eqref{eq:assnoflux} is satisfied by Assumption~\ref{ass: conservation of mass}. For condition~\eqref{eq:assPositivity}, we note that $\nabla \cdot \bi{A}=\omega \nabla \cdot \frac{\rho \bi{u}}{(\rho^{1/2} c)^2}=\frac{2 \omega}{\rho^{1/2} c} \rho \bi{u} \cdot \nabla \frac{1}{\rho^{1/2} c}$, where we have used \eqref{ass: conservation of mass}. Therefore,
\begin{align*}
    \diver \bi{A}-\Im k^2 +\Im v =-\frac{2 \gamma \omega}{c^2} \leq 0.\qquad\qquad \qedhere
\end{align*}
\end{proof}

Because of Assumption~\eqref{eq: parameter_exterior}, $c, \rho, \nabla \rho, \gamma, \bi{u}, \nabla \bi{u}$ are fixed at $\partial \Omega_0$ and in the exterior. Therefore, the space of parameter perturbations is
\begin{align*}
    \mathbb{X}_{\mathcal{P}} := W_0^{1, \infty}(\Omega_0,\R) \times W_0^{2, \infty}(\Omega_0, \R) \times L_0^{\infty}(\Omega_0) \times W_0^{\infty}(\diver, \Omega_0) \times L^{\infty}(\Omega_0)
\end{align*}
where $W_0^{\infty}(\diver, \Omega_0)=\{\bi{u} \in W^{\infty}(\diver, \Omega_0): \bi{u} \cdot \bi{n} = 0 \text{ on } \partial \Omega\}$. 
 \begin{lemma}
     For $c_{\text{min}}, \rho_{\text{min}}>0$ and $k \in \C$, the operator $\mathcal{P}$ is Fr\'echet differentiable in the interior of $\mathcal{A}_{c_{\text{min}}, \rho_{\text{min}}, k}$ with Fr\'echet derivative
     $\mathcal{P}':\mathbb{X}_{\mathcal{P}}\to \mathbb{X}_{\mathcal{G}}$ given by 
     \begin{align}
\label{eq: frechet P}
    \mathcal{P}^{\prime}[c, \rho, \gamma, \bi{u}, S](\partial c, \partial \rho, \partial \gamma, \partial \bi{u}, \partial S)=\begin{pmatrix}
\sum_{q \in \{c, \rho, \gamma, \bi{u} \}} (\partial_q v)(\partial q)\\
\sum_{q \in \{c, \bi{u} \}} (\partial_q \bi{A})(\partial q) \\
\partial S
\end{pmatrix}.
\end{align}
For arguments $(\tilde{v}, \tilde{\bi{A}}, \tilde{S}) \in L^{1}(\Omega_0;\C)\times L^{1}(\Omega_0;\Rd) \times 
L^{1}(\Omega_0;\R)\subset \mathbb{X}^{\prime}$, the values of the adjoint 
\begin{align*}
    \mathcal{P}^{\prime}[c, \rho, \gamma, \bi{u}, S]^*(\tilde{v}, \tilde{\bi{A}}, \tilde{S})=\begin{pmatrix}
    (\partial_c v)^* \tilde{v}+(\partial_c \bi{A})^* \tilde{\bi{A}} \\
    (\partial_{\rho} v)^* \tilde{v} \\
    (\partial_{\gamma} v)^* \tilde{v} \\
    (\partial_{\bi{u}} v)^* \tilde{v}+(\partial_{\bi{u}} \bi{A})^* \tilde{\bi{A}} \\
    \tilde{S}
    \end{pmatrix}
\end{align*}
belong to 
$W^{-2, 1}(\Omega_0, \R) \times W^{-1, 1}(\Omega_0, \R) \times L^{1}(\Omega_0, \R) \times L^1(\Omega_0, \Rd) \times L^{1}(\Omega_0, \R)
\subset \mathbb{X}_{\mathcal{P}}^{\prime}$.
\end{lemma}

\begin{proof}
    We rephrase the potential $v$ in the form:
\begin{eqnarray*}
    v=k^2-\frac{\omega^2+2i \gamma \omega}{c^2}+\rho^{1/2} \Delta (\rho^{-1/2})-i \omega \nabla \cdot \left(\frac{\bi{u}}{c^2} \right).
\end{eqnarray*}
It follows that
\begin{equation}
\label{eq: partial derivatives}
\begin{split}
    &\fl [\partial_c v] (\partial c)=2 \frac{\omega^2+2 i \omega \gamma}{c^3} \cdot \partial c+ 2 i \omega \nabla \cdot \left( \frac{\bi{u}}{c^3} \partial c \right)=:M_{g_c^0}(\partial c)+\left( M_{\bi{g}_c^1} \circ \nabla \right) (\partial c)\\
    &\fl [\partial_{\gamma} v](\partial \gamma)=-2i \omega M_{\frac{1}{c^2}}(\partial \gamma)=:M_{g_{\gamma}^0}(\partial \gamma) \\
    &\fl [\partial_{\rho} v](\partial \rho)=\left( \frac{1}{2} \rho^{1/2} \Delta \rho^{-3/2}-\frac{1}{2} \rho^{-1/2} \Delta \rho^{-1/2} \right) \cdot \partial \rho-\frac{1}{2} \rho^{-1} \Delta \partial \rho-\rho^{1/2} \nabla \rho^{-3/2} \cdot \nabla \partial \rho\\
    &\fl \hspace{2cm} =:M_{g^0_\rho}(\partial \rho)+\left( M_{\bi{g}_{\rho}^1} \circ \nabla \right) (\partial \rho)+\left( M_{g^2_{\rho}} \circ \Delta \right) (\partial \rho) \\
    &\fl [\partial_{\bi{u}} v](\partial \bi{u})=-i \omega \nabla \left( \frac{\partial \bi{u}}{c^2} \right)=:M_{\bi{g}^0_{\bi{u}}}(\partial \bi{u})+\left( M_{g^1_{\bi{u}}} \circ \nabla \right) (\partial \bi{u}),
\end{split}
\end{equation}
where
\begin{eqnarray*}
    &g_c^0,g_{\gamma}^{0}, g_{\rho}^0 \in L^{\infty}(\Omega_0), \quad \bi{g}_{\bi{u}}^0 \in L^{\infty}(\Omega_0)^d, \quad  g_{\bi{u}}^1 \in W^{1, \infty}(\Omega_0), \\
    &\bi{g}_c^1, \bi{g}_{\rho}^1 \in  W^{1, \infty}(\Omega_0)^d, \quad  g_{\rho}^2 \in W^{2, \infty}(\Omega_0).
\end{eqnarray*}
Furthermore, we have 
\begin{eqnarray}
\label{eq: partial derivatives_2}
    &\partial_c \bi{A}=M_{\frac{-\omega}{c^3} \bi{u}}, \qquad \partial_{\bi{u}} \bi{A}=M_{\frac{\omega}{c^2}}, \qquad \partial_{\gamma} \bi{A}=\partial_{\rho} \bi{A}=0,
\end{eqnarray}
The operator $\mathcal{P}$ is Fr\'echet differentiable with Fr\'echet derivative~\eqref{eq: frechet P} since the terms $\partial_q v, \partial_q \bi{A}$ are well-defined for $q \in \{c, \rho, \gamma, \bi{u} \}$. The claim follows with the mapping properties of $(\partial_q v)^*, (\partial_q \bi{A})^*$.
\end{proof}

In analogy to equation~\eqref{eq: Frechet-derivative_propagator}, we can write the Fr\'echet derivative in the form:

\begin{equation}
\label{eq: Frechet-derivative_propagator_2}
\begin{split}
&(\mathcal{C} \circ \mathcal{P})^{\prime}[c, \rho,\gamma,\bi{u},S](\partial c, \partial \rho, \partial \gamma, \partial\bi{u},\partial S)=\sum_{q \in \{c, \rho, \gamma, \bi{u}, S \}} \Re \left( \mathcal{H}^{v, \bi{A}}_{\alpha_q} \mathcal{L}_q(\partial q) {\mathcal{H}^{v, \bi{A}}_{\beta_q}}^* \right)\\
& (\mathcal{C}\circ \mathcal{P})^{\prime}[c, \rho,\gamma,\bi{u},S]^*\mathcal{D}
= \begin{pmatrix}
\mathcal{L}_c^*\left({\mathcal{H}^{v, \bi{A}}_{\alpha_c}}^* \Re(\mathcal{D}) \mathcal{H}^{v, \bi{A}}_{\beta_c}\right)\\
\mathcal{L}_{\rho}^*\left({\mathcal{H}^{v, \bi{A}}_{\alpha_{\rho}}}^* \Re(\mathcal{D}) \mathcal{H}^{v, \bi{A}}_{\beta_{\rho}}\right)\\
\mathcal{L}_{\gamma}^*\left({\mathcal{H}^{v, \bi{A}}_{\alpha_{\gamma}}}^* \Re(\mathcal{D}) \mathcal{H}^{v, \bi{A}}_{\beta_{\gamma}}\right)\\
\mathcal{L}_{\bi{u}}^*\left({\mathcal{H}^{v, \bi{A}}_{\alpha_{\bi{u}}}}^* \Re(\mathcal{D}) \mathcal{H}^{v, \bi{A}}_{\beta_{\bi{u}}}\right)\\
\mathcal{L}_{S}^*\left({\mathcal{H}^{v, \bi{A}}_{\alpha_{S}}}^* \Re(\mathcal{D}) \mathcal{H}^{v, \bi{A}}_{\beta_{S}}\right)
\end{pmatrix}.
\end{split}    
\end{equation}

The operators $\mathcal{L}_q$ play the role of local correlation operators. The propagators and local correlation operators in the flow-free case can be read in Table~\ref{tab: Backpropagator}.

\begin{table}[hb]
\caption{\label{tab: Backpropagator}Distributional kernel of back-propagator and local correlation operator for the different parameters. The functions $g^0_\rho, g^1_\rho, g^2_\rho$ are defined in equation~\eqref{eq: partial derivatives}. The coordinates are chosen such that $\bi{x} \in \Gamma$ and $\bi{y} \in \Omega$. Here, 
we assume that $\bCov[s_{\partial\Omega}] = M_{B}$ with 
$B\in L^{\infty}(\partial\Omega)$ and 
use the notation $S_{\overline{\Omega}}:= S + B \delta_{\partial \Omega}$.}
\footnotesize\rm
\begin{tabular}{@{}*{7}{l}}
\br                              
Quantity $q$ &Propagator $\mathcal{H}_{\alpha_q}$&Propagator $\mathcal{H}_{\beta_q}$&Local correlation $\mathcal{L}_q^*$\cr 
\mr
Source Strength $S$&$G(\bi{x}, \bi{y})$&$G(\bi{x}, \bi{y})$&$\Diag $\cr
Sound speed $c$&$G(\bi{x}, \bi{y})$&$\int_{\Omega} G(\bi{x}, \bi{z}) S_{\overline{\Omega}}(\bi{z})\overline{G(\bi{y}, \bi{z})}\,\mathrm{d} \bi{z}$&$-2 \frac{\omega^2+2i\omega \gamma}{c^3} \cdot \Diag$ \cr
Density $\rho$&$G(\bi{x}, \bi{y})$&$\int_{\Omega} G(\bi{x}, \bi{z}) S_{\overline{\Omega}}(\bi{z})\overline{G(\bi{y}, \bi{z})}\,\mathrm{d} \bi{z}$& $\left( g^0_\rho-\bi{g}_{\rho}^1 \nabla+g^2_{\rho} \Delta \right) \cdot \Diag$ \cr
Wave damping $\gamma$ &$G(\bi{x}, \bi{y})$&$\int_{\Omega} G(\bi{x}, \bi{z}) S_{\overline{\Omega}}(\bi{z})\overline{G(\bi{y}, \bi{z})} \,\mathrm{d} \bi{z}$&$\frac{2 i \omega}{c^2} \cdot \Diag$ \cr
Flow component $A_i$ &$G(\bi{x}, \bi{y})$&$\mathbf{\hat{e}}_i \cdot \nabla_{\bi{y}} \left( \frac{\int_{\Omega} G(\bi{x}, \bi{z}) S_{\overline{\Omega}}(\bi{z}) \overline{G(\bi{y}, \bi{z})} \,\mathrm{d} \bi{z}}{\rho^{1/2}(\bi{y}) c(\bi{y})}\right)$&$2 i \omega \frac{\rho^{1/2}}{c} \cdot \Diag$ \cr
\br
\end{tabular}
\end{table}

Note that despite the fact that the adjoint with respect to the standard $L^2$ dual pairings takes values in negative Sobolev spaces, it is usually not necessary to deal with such functions (or distributions) numerically in iterative regularization methods. For instance, in Landweber iteration in Banach spaces, the application of the adjoint is followed by the application of a 
duality mapping which takes values in positive Banach spaces. For Hilbert space methods, 
one would choose a $L^2$-based Sobolev space $W^{s,2}$ with sufficiently large $s$ and 
compute the adjoint with respect to the $W^{s,2}$ inner product, which amounts to an 
evaluation of the adjoint of the embedding $W^{s,2}\hookrightarrow L^2$.

\subsection{Source model in helioseismology}
\label{sec: seismic sources}
It remains to discuss the seismic source model in helioseismology. 
It has been shown in several settings that the cross-correlation is roughly linked to the imaginary component of the outgoing Green's function (\cite{Garnier2016}). In helioseismology, this relation takes the form (\cite{Gizon2017})
\begin{eqnarray} \label{eq: imaginary_basic}
    C(\bi{r}_1, \bi{r}_2, \omega)=\frac{\Pi(\omega)}{4 i \omega} \left( G_{v, \bi{A}}(\bi{r}_1, \bi{r}_2, \omega)-\overline{G_{v, -\bi{A}}(\bi{r}_1, \bi{r}_2, \omega)} \right),
\end{eqnarray}
where $\Pi(\omega)$ is the source power spectrum. This relation leads to a power spectrum in good agreement with the observations \cite{Gizon2017}. 
As outlined in \cite{Gizon2017}, Equation~\eqref{eq: imaginary_basic} holds true for an outgoing radiation condition and random sources that are appropriately excited across the volume in proportion to the damping rate
\begin{eqnarray}
    \label{eq: source covariance}
    (\bCov[s] \phi)(\bi{r}, \omega)=\Pi(\omega) \frac{\gamma(\bi{r}, \omega)}{c_0^2(\bi{r})} \phi(\bi{r}, \omega).
\end{eqnarray}
Moreover, there are surface integrals that persist for frequencies above the acoustic cutoff frequency, and these are dependent on the chosen boundary condition.

 The relationship between source power and damping rate emerges from the idea of equipartition among distinct acoustic modes (\cite{Snieder2007}). This choice of covariance couples the source strength with wave attenuation and sound speed. Nevertheless, we consider the source strength as an additional individual parameter. This source model is included in the discussion of the previous sections. In helioseismology, the relation~\eqref{eq: imaginary_basic} is the standard choice to reduce the computational costs of the operator evaluation. Furthermore, it allows us to evaluate the back-propagator in Table~\ref{tab: Backpropagator} efficiently.

\section{Iterative helioseismic holography}
\label{sec: applications holography}
In this section, we discuss the application of the approach outlined in Section~\ref{sec: Algorithm} to local helioseismology. We first show that it can be interpreted as an extension of conventional helioseismic holography. 
For this reason, we will refer to this approach as iterative helioseismic holography. We also discuss relations to other methods in local helioseismology. 

In a second subsection, we will describe sensitivity kernels for 
the normal equation as introduced in \eqref{eq:defi_sensitivity_kernel} 
for the following three scenarios: 
\begin{enumerate}
    \item Inversion for the source strength,
    \item Inversion for scalar parameter $q \in \{\rho, c, \gamma\}$,
    \item Inversion for mass-conserved flow field $\bi{u}$.
\end{enumerate} 

\subsection{Relations to conventional helioseismic holography and other methods}
\label{sec: iterative helioseismic holography}

 Conventional helioseismic holography is based on the Huygens principle in the sense that the observed wavefield is described as a superposition of seismic point sources on the wavefront. This principle allows holography to propagate the correlations of acoustic waves at the solar surface forward in time ("ingression" using $\mathcal{H}_{\beta}$) or backward in time ("egression" using $\mathcal{H}_{\alpha}$) to a pre-defined target location in the interior in order to image anomalies in the background medium (e.g. \cite{LB90}). There exists a close connection to seismic migration in terrestrial seismology, which re-locates seismic events on the earth's surface in time and space, based on the wave equation (e.g. \cite{Hagedoorn1954}, \cite{Claerbout1985}). Furthermore, similar back-propagators are used in conventional beamforming in aeroacoustics (\cite{Garnier2016}, \cite{Hohage2020}).

The Lindsey-Braun holographic image (see \cite{LB2000}) is constructed by the wave propagators 
$\mathcal{H}_{\alpha} \in L\left(H_0^{-1}(\Omega), L^2(\Gamma_1)\right)$ 
and $\mathcal{H}_{\beta}\in L\left(H_0^{-1}(\Omega), L^2(\Gamma_2)\right)$ 
such that 
\begin{equation*}
    \phi_{\alpha}(\bi{x})=(\mathcal{H}_{\alpha}^*\psi)(\bi{x})
    =\int_{\Gamma_1} H_{\alpha}(\bi{x}, \bi{r}) \psi(\bi{r}) \,\mathrm{d} \bi{r}, \qquad 
    \phi_{\beta}(\bi{x})=(\mathcal{H}_{\beta}^*\psi)(\bi{x})=\int_{\Gamma_2} H_{\beta}(\bi{x}, \bi{r}) \psi(\bi{r}) \,\mathrm{d} \bi{r},
\end{equation*}
where $\Gamma_1, \Gamma_2 \subset \Gamma$ are called pupils.
In Lindsey-Braun holography the information is extracted from the so-called egression-ingression correlation for parameters $q \in \{c, \rho, \bi{u}, \gamma\}$ and the egression power for seismic sources 
\begin{eqnarray}
    \label{eq: hologramintensity}
    \fl I_{\alpha, \beta} (\bi{x})=\frac{1}{N} \sum_{n=1}^N \phi^n_{\alpha}(\bi{x}) \overline{\phi^n_{\beta}(\bi{x})} =
    \frac{1}{N} \sum_{n=1}^N \Diag \left( \mathcal{H}^*_{\alpha} \psi_n \otimes \mathcal{H}^*_{\beta} \psi_n \right)(\bi{x}), \\
    \label{eq: hologramintensity_2}
    \fl \mathbb{E}_{\bi{q}} \left[ I_{\alpha, \beta} (\bi{x}) \right]=\int_{\Gamma} \int_{\Gamma} \overline{H_{\alpha} (\bi{r}, \bi{x})} C_{v,\bi{A}, S}(\bi{r}, \bi{r}_1) H_{\beta} (\bi{r}_1, \bi{x}) \,\mathrm{d} \bi{r} \,\mathrm{d} \bi{r}_1=\Diag \left( \mathcal{H}^*_{\alpha} 
    \mathcal{C}[\bi{q}] \mathcal{H}_{\beta} \right),
\end{eqnarray}
where  $\otimes$ the standard tensor product, and $C_{v,\bi{A}, S}$ is from equation~\eqref{eq: Covariance}. 

The comparison of equations~\eref{eq: Frechet-derivative_propagator} and \eref{eq: hologramintensity} shows that the adjoint of the Fr\'echet derivative of the covariance operator is linked to traditional helioseismic holography. Denoting potential additional dependence 
of $I_{\alpha,\beta}$ on the unknown parameters $\bi{q}$ through $\mathcal{H}_{\alpha}$ and $\mathcal{H}_{\beta}$ 
by superscripts, in terms of conventional holography the Newton step~\eref{eq: Iteration} is a regularized solution to
\begin{eqnarray*}
    I_{\alpha, \beta}^{\bi{q}_n}-\mathbb{E}_{\bi{q}} \left[ I_{\alpha, \beta}^{\bi{q}_n} \right] =\int_{\Omega_0} \bi{K}_{\alpha, \beta}^{\bi{q}_n}(\cdot,\bi{y})\,(\delta \bi{q}_n)(\bi{y})\,\mathrm{d} \bi{y},
\end{eqnarray*}
where $\bi{K}_{\alpha, \beta}^{\bi{q}}$ are the sensitivity kernel of traditional holography, see \eqref{eq:forwardbackward}. 
We will discuss the sensitivity kernels in more detail in Section~\ref{sec: sensitivity_Kernel}.

In traditional helioseismic holography, one has freedom in the choice of the pupils and back-propagators.
For example, the pupils can be chosen such that the hologram intensity becomes sensitive to specific flow components \cite{Yang2018b}. 
As a further example, Porter-Bojarski holograms, introduced to the field of helioseismology in \cite{Skartlien2001, Skartlien2002}, make use of the normal derivative at the surface in addition to the Dirichlet data.
In contrast, the backward propagators in iterative helioseismic holography are determined by the wave equation, and the image is improved by iteration.

While many techniques in helioseismology including traditional helioseismic holography are limited to linear scenarios, iterative holography naturally allows to tackle nonlinear problems. 

Among the commonly used imaging techniques in local helioseismology, holography is the only method that uses the complete cross-correlation data. 
As already discussed in Section~\ref{sec: Avoid Corr}, these data 
are used only in an implicit manner without the (usually infeasible) 
requirement of computing or storing the cross-correlation data explicitly. 
Whereas traditional helioseismic holography only provides feature maps 
 \cite{LB97}, iterative helioseismic holography additionally allows to retrieve quantitative information.


\subsection{Kernels and resolution}
\label{sec: sensitivity_Kernel}

In the following, we compute the sensitivity kernels for the normal equation as defined in eq.~\eqref{eq:defi_sensitivity_kernel} which are set up explicitly in our current implementation. 
It is important to note that sensitivity kernels are typically 3D $\times$ 3D operators and should be avoided in computations. Nevertheless, in the spherically symmetric case or two-dimensional medium, computation becomes feasible. Therefore, for the purpose of this paper, we can compute the sensitivity kernels in each iteration and do not study more sophisticated approaches.
These kernels are infinitely smooth for smooth coefficients, but they are well localized. It turns out that the width of these kernels is of the order of the classical resolution limit of half a wavelength. 
This provides an upper bound on the achievable resolution. For simplicity, we will assume a spherically symmetric background without a flow field.

\begin{enumerate}
    \item 
\emph{Inversion for source strength:} It follows from Theorem~\ref{the: HS-operator} that  
\begin{eqnarray*}
    \partial_S C[v, \bi{A}, S](\partial S)=H_{\alpha_S} \mathcal{M}(\partial S) H^*_{\beta_S},
\end{eqnarray*}
where the multiplication operator is defined in Lemma~\ref{lem:MS}, so
\begin{eqnarray*}
    (\partial_S C[v, \bi{A}, S])^{*} \partial_S C[v, \bi{A}, S](\partial S)=\Diag(\mathcal{F}_{\alpha_S, \alpha_S} \mathcal{M}(\partial S) \mathcal{F}_{\beta_S, \beta_S})
\end{eqnarray*}
with the sensitivity kernel 
\begin{eqnarray*}
    K(\bi{x},\bi{y})=\Re \left[ F_{\alpha_S, \alpha_S}(\bi{x}, \bi{y}) F_{\beta_S, \beta_S}(\bi{y}, \bi{x}) \right].
\end{eqnarray*}
The real part comes from the fact that the source strength has to be a real parameter, it is the adjoint of the embedding of 
a vector space of real-valued functions into the corresponding 
vector space of complex-valued functions.
The source forward-backward kernel takes the form:
\begin{eqnarray*}
F_{\alpha_S, \alpha_S}(\bi{x}, \bi{y})=F_{\beta_S, \beta_S}(\bi{y}, \bi{x})=\int_{\Gamma} \int_{\Gamma} \overline{G(\bi{z}, \bi{x})} D_n(\bi{z}, \bi{z}^{\prime}) G(\bi{z}^{\prime}, \bi{y}) \,\mathrm{d} \bi{z} \,\mathrm{d} \bi{z}^{\prime},
\end{eqnarray*}
where $D_n$ is the integral kernel of $\Gamma_n$ in \eqref{eq:defi_Gamma_n}.  
The sensitivity kernel becomes $|K_{\alpha_S, \alpha_S}|^2$ and is therefore non-negative. 
Furthermore, there are almost no sidelobes after averaging over frequency.

\item \emph{Inversion for scalar parameters $q\in\{\rho,c,\gamma\}$:}
The operators $\partial_q v$ and $\partial_q \bi{A}$ for 
 $q \in \{\rho, c, \gamma\}$ are computed in equations \eqref{eq: partial derivatives} and \eqref{eq: partial derivatives_2}. For a flow-free background medium, we have $\partial_q \bi{A}=0$ for all scalar parameters $q$. It follows from Theorem~\ref{the: HS-operator} that
\begin{eqnarray*}
    \partial_q C[v, \bi{A}, S](\partial q)=-2 \Re \left[H_{\alpha_v} \mathcal{M}(\partial_q v \partial q) H^*_{\beta_v} \right], 
\end{eqnarray*}
and hence 
\begin{eqnarray*}
    &\fl (\partial_q C[v, \bi{A}, S])^{*} \partial_q C[v, \bi{A}, S](\partial q)\\
    &=2 (\partial_q v)^* \Diag \left[ \mathcal{F}_{\alpha_v, \alpha_v} \mathcal{M}(\partial_q v \partial q) \mathcal{F}_{\beta_v, \beta_v}+ \mathcal{F}_{\alpha_v, \beta_v} \mathcal{M}((\partial_q v \partial q)^*) \mathcal{F}_{\alpha_v, \beta_v} \right],
\end{eqnarray*}
where $F_{\alpha_{v}, \alpha_{v}}=\int_{\Gamma} \int_{\Gamma} \overline{H_{\alpha_{v}}(\bi{z}, \bi{x})} D_n(\bi{z}, \bi{z}^{\prime}) H_{\alpha_{v}}(\bi{z}^{\prime}, \bi{y}) \,\mathrm{d} \bi{z} \,\mathrm{d} \bi{z}^{\prime}$ and analogue for $F_{\alpha_v, \alpha_v}, F_{\beta_v, \beta_v}$.
In particular, for $q \in \{c, \gamma\}$, we have $\partial_q v=\mathcal{M}(g_q^0)$,  and the kernel takes the form
\begin{equation}\label{eq:kernel_scalar}
\begin{aligned}
     K(\bi{x}, \bi{y})&=2 \Re \left[g_q^{0*} (\bi{x}) F_{\alpha_v, \alpha_v}(\bi{x}, \bi{y}) F_{\beta_v, \beta_v}(\bi{y}, \bi{x}) g_q^0 (\bi{y}) \right] \\
    \fl &+ 2 \Re \left[g_q^{0*} (\bi{x}) F_{\alpha_v, \beta_v}(\bi{x}, \bi{y}) F_{\alpha_v, \beta_v}(\bi{y}, \bi{x}) g_q^{0*}(\bi{y}) \right].
    \end{aligned}
\end{equation}

\item \emph{Inversion for mass-conserved flow field $\bi{u}$:}
The flow field sensitivity kernel takes the form 
\begin{eqnarray*}
    \fl K^{i, j} (\bi{x}, \bi{y})=2 \Re  \left[ F_{\alpha_{\bi{u}}, \alpha_{\bi{u}}}(\bi{x}, \bi{y}) F^{i, j}_{\beta_{\bi{u}}, \beta_{\bi{u}}}(\bi{y}, \bi{x})+ F^{j}_{\alpha_{\bi{u}}, \beta_{\bi{u}}}(\bi{x}, \bi{y}) F^i_{\alpha_{\bi{u}}, \beta_{\bi{u}}}(\bi{y}, \bi{x}) \right],
\end{eqnarray*}
for $i, j \in \{r, \theta, \phi\}$ 
where $F_{\alpha_{\bi{u}}, \alpha_{\bi{u}}}=\int_{\Gamma} \int_{\Gamma} \overline{H_{\alpha_{\bi{u}}}(\bi{z}, \bi{x})} D_n(\bi{z}, \bi{z}^{\prime}) H_{\alpha_{\bi{u}}}(\bi{z}^{\prime}, \bi{y}) \,\mathrm{d} \bi{z} \,\mathrm{d} \bi{z}^{\prime}$ and analogue the further kernels.
\end{enumerate}

It is remarkable that we are encountering a gradient in the local correlation $(\partial_{\bi{u}} L_{\bi{u}})^{*}$. In Chapter~4 of \cite{Yang2018b}, enhancements were accomplished by calculating the difference between two holograms which are spaced apart by half the local wavelength. This can be understood as an approximation to the gradient in the target direction and is therefore naturally incorporated in our framework. 

\begin{figure}
   \centering
   \includegraphics[width=\hsize]{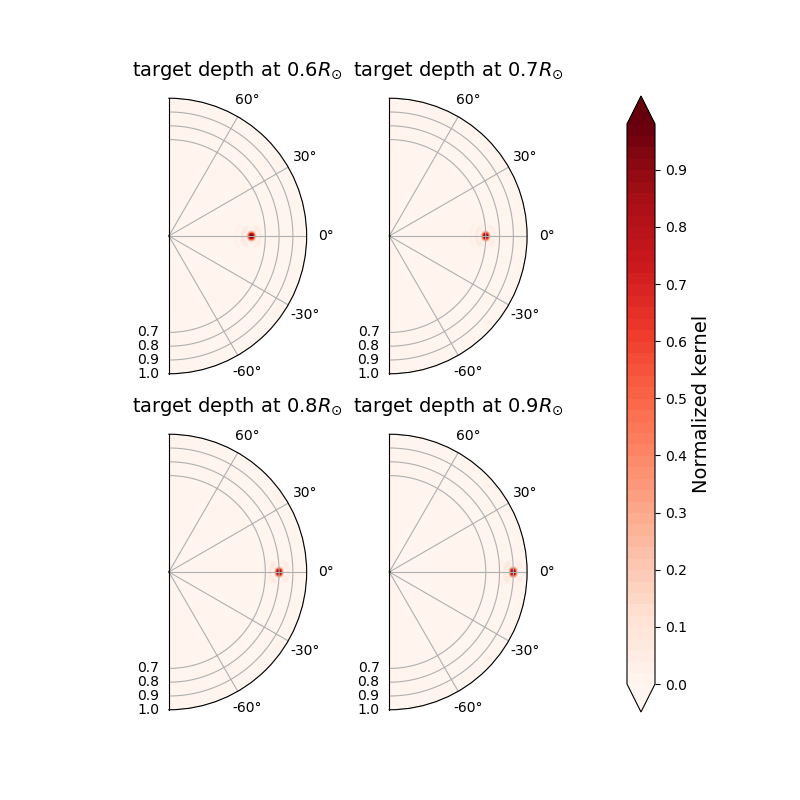}
      \caption{The sound speed sensitivity kernel $K(\bi{x},\cdot)$ in the $r-\theta$-plane as defined in \eqref{eq:kernel_scalar}
      for a three-dimensional uniform medium with $c_0=200$\,km/s and the $l$-range is $0 \leq l<100$ for four different target positions $\bi{x}$. We have averaged the sensitivity kernels over 100 frequencies in the frequency regime $2.75$--$3.25$~mHz and normalized with $K(\bi{x}, \bi{x})$ at the target location $\bi{x}$.
              }
         \label{fig: sound speed kernel_Uniform}
   \end{figure}

\begin{figure}
   \centering
   \includegraphics[width=\hsize]{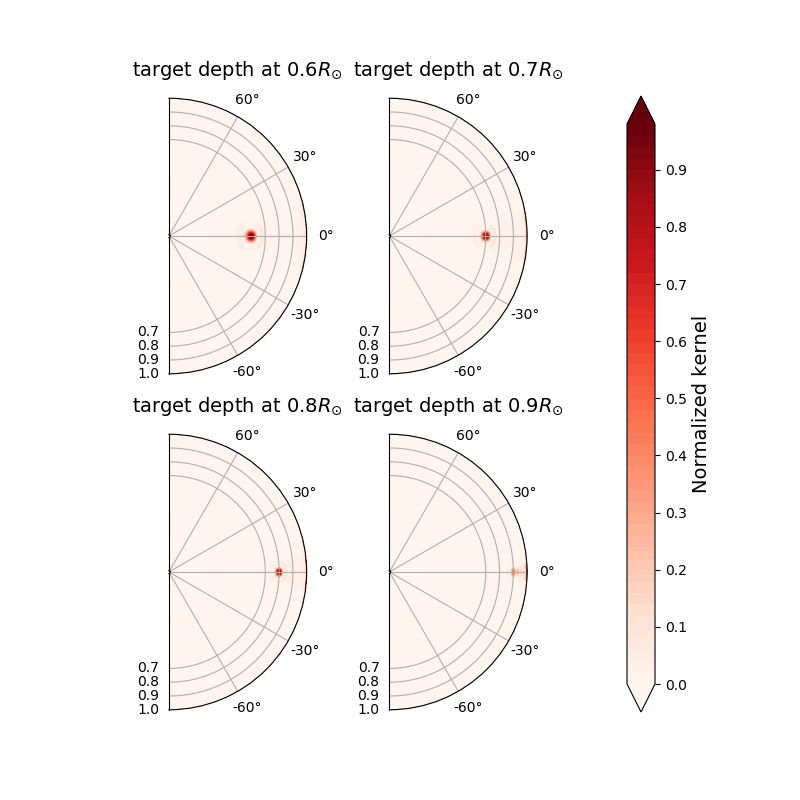}
      \caption{The sound speed sensitivity kernel $K(\bi{x},\cdot)$ in the $r-\theta$-plane as defined in \eqref{eq:kernel_scalar} in a spherically stratified solar-like background medium and 
      spherical harmonics degrees $0 \leq l<100$ for four  different target positions. We have averaged the sensitivity kernels over 100 frequencies in the frequency regime $2.75$--$3.25$~mHz and normalized with $K(\bi{x}, \bi{x})$ at the target location $\bi{x}$. For better comparisons, we have multiplied the sensitivity kernels with the sound speed.
              }
         \label{fig: sound speed kernel_Sun}
   \end{figure}
   
\begin{figure}
   \centering
   \includegraphics[width=\hsize]{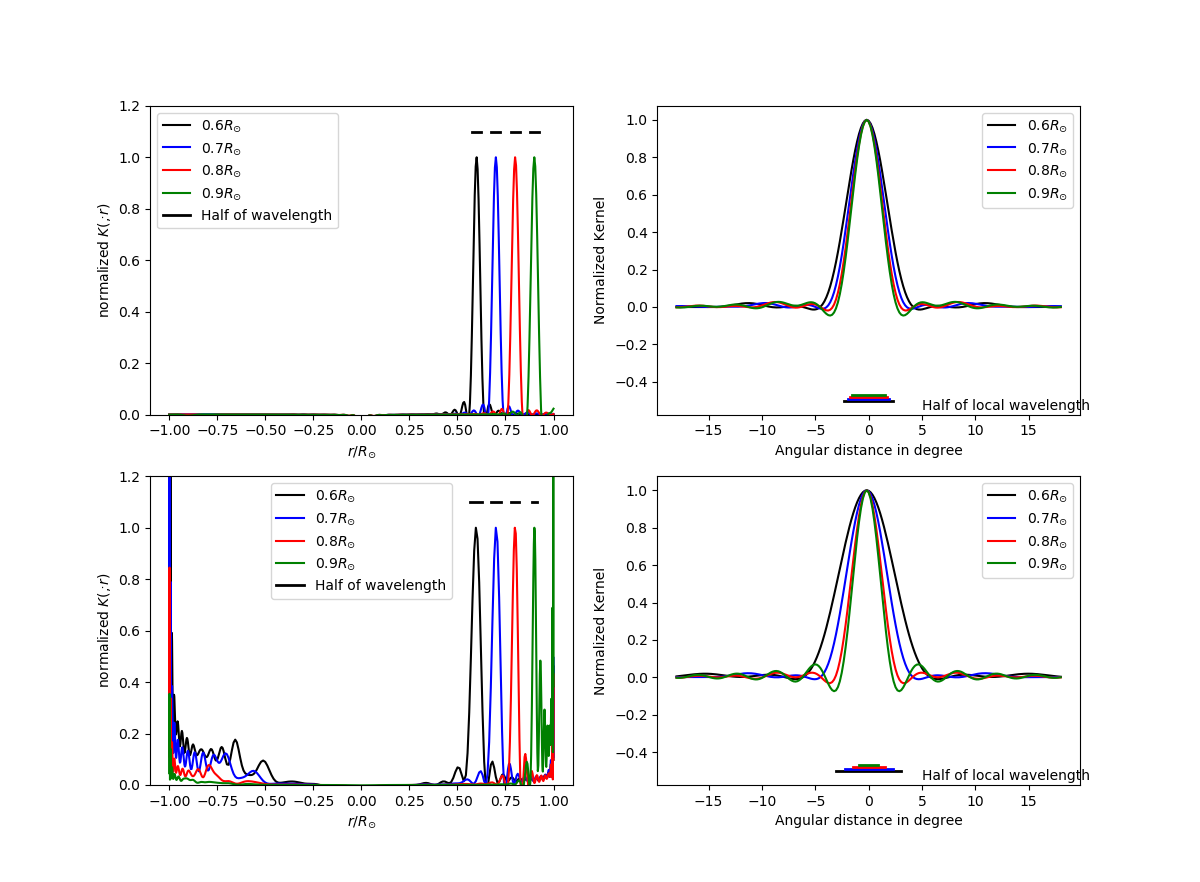}
      \caption{The sound speed sensitivity kernels in a spherically stratified background medium and the $l$ range is $0 \leq l<100$. In the top panels we present the kernels for a uniform medium with $c_0=200$\,km/s (as in Fig.~\ref{fig: sound speed kernel_Uniform}), 
      and in the second line the kernels for a solar-like 
      medium (as in Fig.~\ref{fig: sound speed kernel_Sun}). 
      In the first column, we show the kernels in the radial direction, and in the second column the kernels in the angular direction. We have averaged the sensitivity kernels over 100 frequencies in the frequency regime $2.75$--$3.25$~mHz. Furthermore, we compare the width of the sensitivity kernels to the classical resolution limit of $\lambda/2$. For better comparisons, we have multiplied the sensitivity kernels with the sound speed.}
     \label{fig: sound speed kernel}
\end{figure}
   
In Figures~\ref{fig: sound speed kernel_Uniform} and \ref{fig: sound speed kernel_Sun}, we present the sound speed kernel for a uniform and a solar-like radially stratified medium for four different target positions. The kernels are computed for spherical harmonic degrees $0 \leq l<100$ and averaged over 100 evenly spaced frequencies between $2.75$--$3.25$~mHz. Since there are no strong ghost images on the backside, we show only half of the geometry. The sound speed kernels are very sharp near the target location. Therefore, we can expect the holograms to catch the main features of the image. It is important to highlight that the kernels maintain their sharpness even in deep regions within the interior. In addition, this result holds true for a radial stratification similar to that of the Sun. We observe similar behavior for the sensitivity kernels for wave damping, density, source strength, and the components of the flow field. Similar to the sensitivity kernels for the source strength, it is important to note that there are only small visible sidelobes in the sensitivity kernels for sound speed perturbations. This is an additional advantage compared to traditional techniques used in helioseismology.

Figure~\ref{fig: sound speed kernel} provides a comparison of the width of the sensitivity kernels for the normal equation and the local half wavelength $\lambda/2$. Note that both are of similar size in all cases, and similar results hold true in angular direction. Therefore, we can expect a resolution of (at least) $\lambda/2$. However, in the case of a solar-like stratification, the sensitivity kernels are increasing close to the solar surface.

In helioseismology, and particularly in helioseismic holography, a common issue is the indistinguishability of various sources of perturbations, which complicates the interpretation of seismic data. The design of a holographic back-propagator holds the promise of separating different perturbations. In Figure~\ref{fig: sound damping kernel}, we present the sensitivity kernels for a perturbation in sound speed, a perturbation in damping, and the cross-kernel in a uniform two-dimensional medium. We show the kernels in a region around the target location.
Note that the sound speed kernel is on one scale bigger than the damping kernel and the cross-kernel. Furthermore, the cross-kernel exhibits a different shape with positive and negative maxima around the target location. Therefore, we expect that iterative holography can separate different perturbations in the background medium.

\begin{figure}
   \centering
   \includegraphics[width=\hsize]{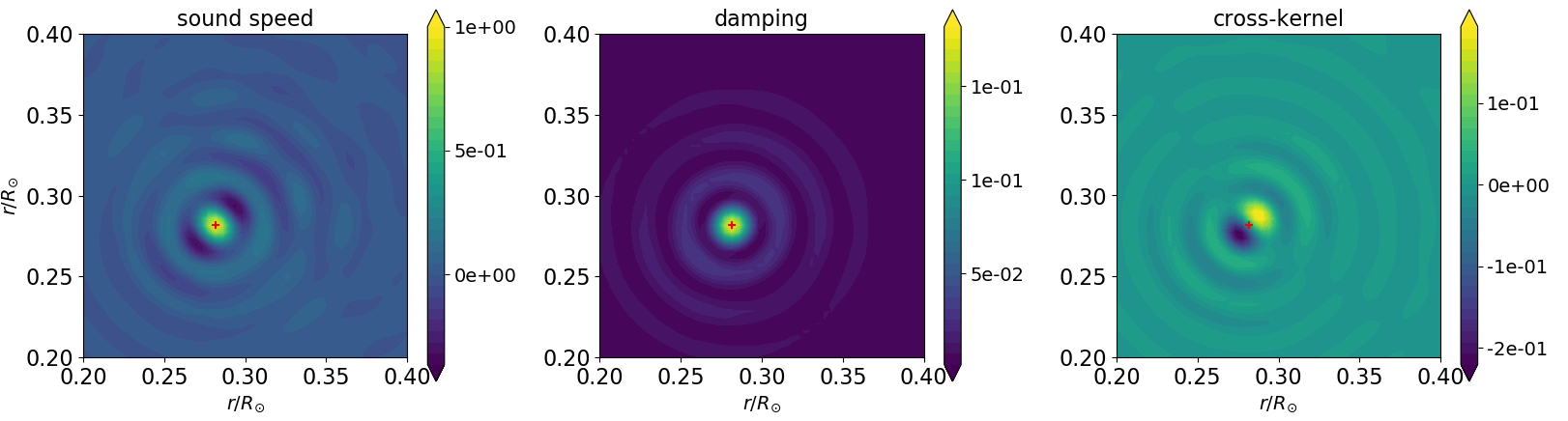}
      \caption{Matrix-valued sensitivity kernel 
      for joint inversion for sound speed $c$  and damping $\gamma$. The left two panels 
      exhibit the diagonal entries, and the right panel the cross-kernel in a uniform two-dimensional medium in a rectangular box of $[0.2\,R_{\odot}, 0.4\,R_{\odot}]^2$. The target location is indexed by a red cross. The kernels are normalized by the maximal value of the sound-speed kernel.}
     \label{fig: sound damping kernel}
\end{figure}

\section{Inversions}
\label{sec: Inversions}
In this section, we analyze the performance of iterative holography. The geometry is meshed with a resolution of 10 internal points per local wavelength. Furthermore, we impose a Sommerfeld boundary condition throughout the inversions.

Throughout the following inversions, we employ a $L^2$-term as the penalty term and introduce a non-negativity constraint for both sound speed and source strength. The regularization parameter is determined by a power law: $\alpha_n=\alpha_0 \cdot 0.9^n$, where $\alpha_0$ represents the maximal eigenvalue of the first iteration. The stopping criterion for the inversions is a version of the discrepancy principle for the 
normal equation, with the noise level determined by the trace of the covariance operator $\mathcal{C}_4$. In more advanced inversions, stopping rules may be investigated in the hologram space. We set a limit of at most 50 inner conjugate gradient steps per Newton step. 
Furthermore, we opt for a spatial resolution of 7 grid points per local wavelength. 

\subsection{Holographic image for source perturbation}
We have performed a numerical test for a uniform, flow-free two-dimensional medium with source region $[0.5, 0.7]^2$ and 100 uniformly sampled receivers located on $\partial B(0,1)$ (see Figure~\ref{fig: holography not quatitative}). We choose a constant sound speed $c=350$~km/s, which corresponds to the solar sound speed at $\approx 0.38 R_{\odot}$. The frequency is fixed to be $\omega/2 \pi=3~$mHz, which corresponds to the solar 5-minute oscillations. In the case of uniform medium, the differential equation simplifies to a Helmholtz equation, such that the Green's function is analytically known (see \ref{sec: Uniform Green}). Note that Lindsey-Braun holography ($\mathcal{H}_{\alpha}=\mathcal{H}_{\beta}=\mathcal{G}$) provides sharp feature maps in the case of small wave damping. 
For stronger wave damping, the quality of these feature maps deteriorates rapidly. Even without damping, the feature maps are not quantitative at all. 
\begin{figure}
\centering
\includegraphics[width=\hsize]{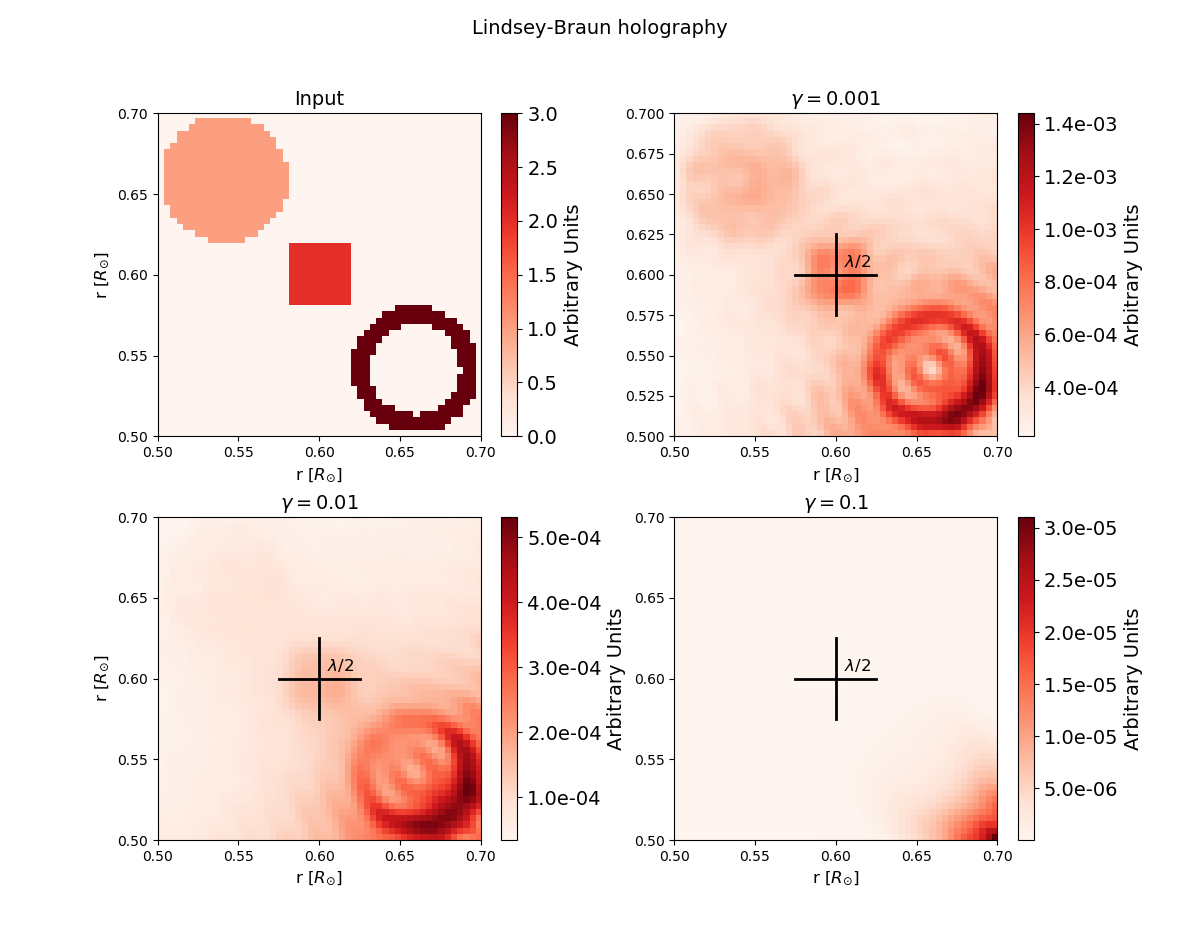}
  \caption{Lindsey-Braun holographic image intensities in uniform two-dimensional medium (Helmholtz equation) for different degrees of damping
  with 100 equidistant receivers on $\partial B(0, 1)$. The wave number is such that $k^2=\omega^2(1+i \gamma)/c^2$ with constant sound speed $c=350$~km/s and $\omega/2 \pi=3$~mHz corresponding to a wavelength of $\approx 0.17\,R_{\odot}$. Note the different scalings of the color maps illustrating the non-quantitative nature of Lindsey-Braun holography.}
     \label{fig: holography not quatitative}
\end{figure}

\subsection{Source strength inversion}
Due to its linear nature, inversion for source strength is the simplest 
case. Therefore, it is in general possible to work with a much finer grid than in the case of parameter identification problems. We add a strong perturbation in the source region $[0, 0.5\,R_{\odot}]^2$. The inversion results at 3~mHz are shown in the first row of Figure~\ref{fig: Inversions_source_sound} for 10000 realizations. 
Note that even very deep source terms can be inverted using only one frequency. 
The reconstructions exhibit a remarkable quality, strongly improving the results by traditional Lindsey-Braun holography (see Figure~\ref{fig: holography not quatitative}).

\begin{figure}
\centering
\includegraphics[width=\hsize]{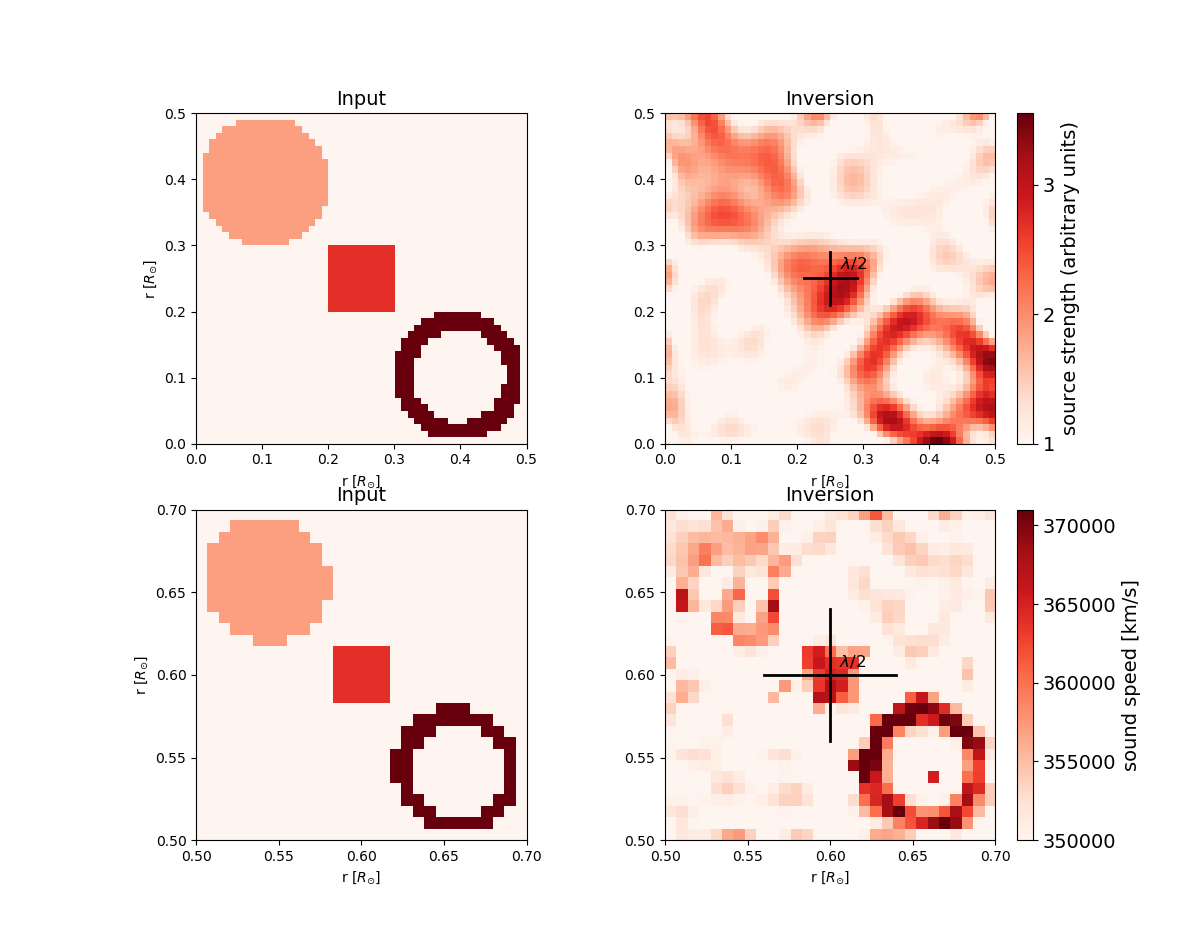}
  \caption{Inversions in two-dimensional uniform background with sound speed $c=350$\,km/s. In the first row, we present inversions for the source strength at 3~mHz and 10000 realizations. In the second row, we present inversions for the sound speed for 100 frequencies evenly spaced in the band $2.75$--$3.25$~mHz and 1000 realizations. The black cross indicates $\lambda/2$.}
     \label{fig: Inversions_source_sound}
\end{figure}

\subsection{Parameter identification}
We add a perturbation in the quadratic region $[0.5R_{\odot}, 0.7 R_{\odot}]^2$.
Furthermore, we choose 100 evenly spaced frequencies in the frequency range of $2.75-3.25$~mHz and assume 1000 realizations for each frequency. Note that in helioseismology we have many more frequencies available.

The inversions are shown in the second row of Figure~\ref{fig: Inversions_source_sound}.  The Newton iteration was stopped after 15 iterations. The resolution of the reconstruction is again below 
the classical limit of half a wavelength. 
We observed qualitatively similar results in the inversions for wave damping and density.

The total number of Dopplergrams is given by $N_{\omega} \times N_{\text{obs}}$, where $N_{\omega}$ is the number of frequencies and $N_{\text{obs}}$ the number of realizations for each frequency. Note that the total size of Doppler data is fixed by the observation time. 
We observe that a larger number of frequencies leads to better reconstructions. On the other hand, the computational costs scale roughly linearly with the number of frequencies. 
This becomes particularly important for large-scale forward problems like for the Sun. Therefore, the choice of $N_{\omega}$ often is a trade-off between quality of reconstructions and computation time.

\subsection{Flow fields}
\label{sec: flow fields}

\begin{figure}
   \centering
   \includegraphics[width=\hsize]{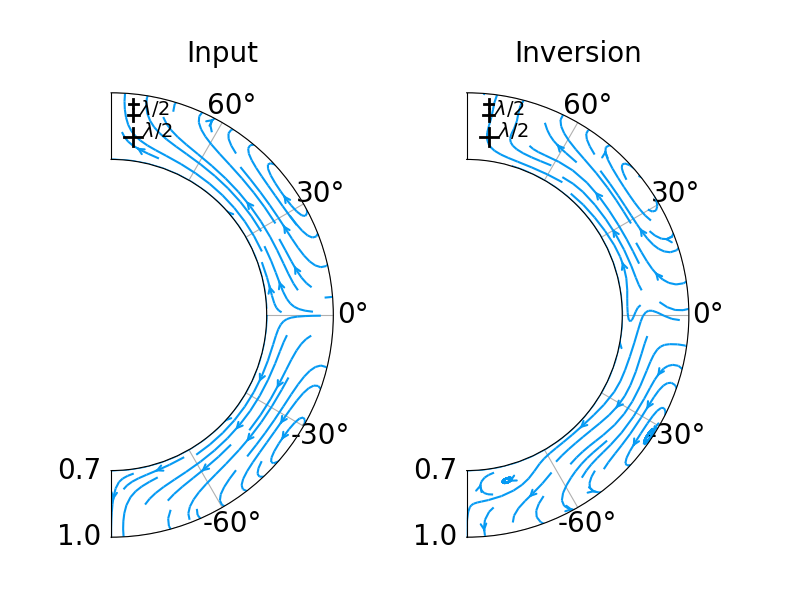}
      \caption{Inversion for the flow field in solar-like three-dimensional background medium in the $r-\theta$-plane. The inversion is performed with 100 evenly spaced frequencies between $2.75$--$3.25$~mHz and $1500$ realizations. Due to the symmetry, we show only one half-space of the flow field.
              }
         \label{fig: flow field inversion}
   \end{figure}

The inversion is performed in a solar-like three-dimensional medium.
The example flow field is computed by $\bi{u}=\operatorname{curl} \psi$, where $\psi$ is a stream function. This guarantees conservation of mass and axisymmetry of the flow field. The stream function is chosen similar to models of meridional circulation profiles in the Sun \cite{Liang2018}. In the inversion process, we guarantee conservation of mass through Lagrange multipliers, as discussed in \ref{sec: conservation of mass}. The inversion for a symmetric flow field is presented in Figure~\ref{fig: flow field inversion}. We inverted with 100 evenly spaced frequencies between $2.75$-$3.25$~mHz and assumed $1500$ realizations for each frequency.  
Since meridional flows are a small perturbation, the iteration is stopped after one iteration.
Because of the symmetry, we show only one half-space of the flow field. Besides the strength of the flow field at larger depths, there is no difference visible in the eye-norm.

\section{Conclusions}
\label{sec: conclusions}
We have developed a theoretical framework for quantitative passive imaging problems in helioseismology. It shows that traditional holography can be interpreted as an adjoint imaging method. Holographic back-propagation can be seen as part of the adjoint of the Fr\'echet derivative of the 
forward operator mapping physical parameters to the 
covariance operator of the observations. In contrast to traditional holography, the backward propagators are uniquely determined by the wave equation, and the holograms can be improved by iteration rather than clever choices of back-propagators. Iterative helioseismic holography surpasses traditional helioseismic techniques by the quantitative nature of its imaging capabilities and its ability to solve nonlinear problems.

We have demonstrated the performance of iterated holography 
in inversions for the right hand side of wave equation
(source strength), 
parameters of the zeroth order term (sound speed, absorption) 
and of the first order term (flows). 
In all three cases, we have achieved reconstructions 
with a resolution of slightly less than 
half of the local wave-length by the iteratively regularized 
Gauss-Newton method, even for strong realization noise. 
This is well below the spatial resolution of traditional time-distance helioseismology (see \cite{Pourabdian2018}).

 Inversions in other more challenging solar setups and for real solar oscillation data are planned as future work and will be presented elsewhere. 

In view of the huge size of solar oscillation data, 
the main bottleneck that prevents the 
immediate application of iterative holography to 
interesting large-scale problems in helioseismology 
is computational complexity. 
The results of this paper encourage further algorithmic research 
on 
iterative regularization methods tailored to passive imaging problems, e.g., by more efficient 
treatments of sensitivity kernels and Green's functions. 

An interesting feature of correlations of Gaussian fields
is the structure of the realization noise as described in Section \ref{sec: Noise model}. A thorough mathematical treatment will 
require further investigation concerning appropriate stopping rules, 
consistency, and convergence rates as the sample size tends to infinity. 

\section*{Acknowledgments}
      This work was supported by the International Max Planck Research School (IMPRS) for Solar System Science at the University of Göttingen. The authors acknowledge partial support from Deutsche Forschungsgemeinschaft (DFG, German Research Foundation) through SFB 1456/432680300 Mathematics of Experiment, project C04.

\section*{References}
\typeout{}
\bibliography{Quantitative_imaging}

\appendix
\section{Reference Green's function for the Sun}
\label{sec: Green}
 The Green's function is usually computed in the frequency domain with background parameters specified by the spherically symmetric standard Model~S \cite{Dalsgaard1996}. Additionally, we have to fix the wave attenuation. We choose a frequency-dependent wave attenuation model, motivated by the Full Width at Half Maximum (FWHM) of wave modes (\cite{Korzennik2013}, \cite{Larson2015}, \cite{Gizon2017}):
\begin{equation*}
    \gamma(\bi{r}, \omega) = \left\{
    \begin{array}{ll}
    \gamma_0 \left| {\omega}/{\omega_0} \right|^{5.77}  &  \textrm{for }\omega \leq 5.3\,\textrm{mHz} 
    \\
     2\pi \times 125~\mu \textrm{Hz} & \textrm{for } \omega \geq 5.3~\textrm{mHz}
     \end{array}
     \right. ,
\end{equation*}
where ${\gamma_0}/{2 \pi}=4.29~\mu$Hz and ${\omega_0}/{2 \pi}=3$~mHz. We extend the computational boundary by 500~km above the solar surface (compare with the density scale height of $H=105$~km) and apply the radiation boundary condition "Atmo Non Local" (see \cite{Fournier2017}), assuming an exponential decay of density and constant sound speed in the solar atmosphere.\\
In a spherical symmetric background, we can decompose the Green's function into spherical harmonics:
\begin{eqnarray}
    G(\bi{r}_1, \bi{r}_2)=\sum_{l=0}^{\infty} \sum_{m=-l}^l G_l(r_1, r_2) Y_{lm}(\hat{\bi{r}}_1) Y_{lm}^*(\hat{\bi{r}}_2),
\end{eqnarray}
where the $Y_{lm}$ are spherical harmonics. The functions $G_l(r_1, r_2)$ satisfy a one-dimensional differential equation and are computed with \textit{NGsolve} \cite{Schoeberl1997}, \cite{Schoeberl2014}. The computation of the Green's function is usually expensive, as the stiffness matrix has to be inverted. The two-step algorithm of \cite{Barucq2020} allows us to obtain the full modal Green's function from two computations only. Furthermore, this expansion allows us to use a low-rank approximation for the Green's function.

\section{Green's function in uniform medium}
\label{sec: Uniform Green}
We perform numerical toy examples in uniform flow-free two-dimensional and three-dimensional background mediums and consider a Sommerfeld boundary condition. The differential equation~\eref{eqn:Helmholtz} reduces to the Helmholtz equation
\begin{eqnarray}
    -(\Delta+k^2) \psi=s,
\end{eqnarray}
where $k$ is constant. In this setting, the Green's function is well known (e.g. \cite{Colton2013}):
\begin{eqnarray}
    &G(\bi{x}, \bi{y}, k)=\frac{i}{4} H^1_0(k \vert \bi{x}-\bi{y} \vert), \quad d=2\\
    &G(\bi{x}, \bi{y}, k)=\frac{\exp(ik \vert \bi{x}-\bi{y} \vert)}{4 \pi \vert \bi{x}-\bi{y} \vert}, \quad d=3,
\end{eqnarray}
where $H^1_0$ is the Hankel function of first kind. 

The Green's functions are weakly singular at $\bi{x}=\bi{y}$. We will approximate the Green's functions around the singularity using asymptotics:
\begin{eqnarray}
\fl G(\bi{x},\bi{y})=\frac{1}{2 \pi} \ln(\frac{1}{\vert \bi{x}-\bi{y} \vert}) + \frac{i}{4}-\frac{1}{2 \pi} \ln(\frac{k}{2})-\frac{C}{2 \pi}+O(\vert \bi{x}-\bi{y} \vert^2 \ln(1/\vert \bi{x}-\bi{y} \vert)), \,\,\,\,d=2\\
\fl G(\bi{x}, \bi{y}) = \frac{1}{4 \pi \vert \bi{x}-\bi{y} \vert}+\frac{ik}{4 \pi}+O(\vert \bi{x}-\bi{y} \vert),\,\,\,\,d=3,
\end{eqnarray}
where the constant $C$ denotes the Euler-Mascheroni constant.
In inversions of extended properties like large-scale flows, it is more feasible to work in an angular basis (spherical harmonics in three dimensions and trigonometric functions in two dimensions). The Green's functions for the uniform medium can be described by
\begin{eqnarray}
&G(\bi{x},\bi{y})=H_0^1(k \vert \bi{x} \vert) J_0(k \vert \bi{y} \vert)+2 \sum_{n=1}^{\infty} H_n^1(k \vert \bi{x} \vert) J_n(k \vert \bi{y} \vert)\cos(n \theta_{x,y}),\,\,\,\,d=2\\
&G(\bi{x}, \bi{y})=ik \sum_{n=0}^{\infty} \sum_{m=-n}^{m=n} h_n^1 (k \vert \bi{x} \vert) Y_{nm}(\hat{\bi{x}}) j_n(k \vert \bi{y} \vert ) Y_{nm}^{*} (\hat{\bi{y}}),\,\,\,\,d=3,
\end{eqnarray}
for $\vert \bi{x} \vert \geq \vert \bi{y} \vert$.
Here, $J_n, h_n^1, j_n$ denote the Bessel function, spherical Hankel function, and spherical Bessel function. Moreover, $\theta_{x,y}$ denotes the angular distance between $\bi{x}$ and $\bi{y}$. Furthermore, this basis transformation allows a natural implementation of the singularity. We use this expansion in order to use low-rank approximations for the Green's function.

\section{Conservation of mass}
\label{sec: conservation of mass}

For flow inversions, considerable improvements are achieved by incorporating mass conservation in the inversion process (\cite{Fournier2018}). 

An equality constraint $R \delta \bi{u}_k=0$, where $R: \mathbb{X} \to \mathbb{Z}$ is a bounded linear operator, can be incorporated by employing the method of Lagrange multiplier. For the iterative Gauss-Newton method, we solve the normal equation:
\begin{eqnarray}
    \fl \delta \bi{u}_k= \underset{\operatorname{div}(\rho \delta \bi{u})=0}{\operatorname{argmin}}  \Vert (\Gamma_n^{1/2} \times \Gamma_n^{1/2}) \left[\mathcal{C}^{\prime}[\bi{u}_k](\delta \bi{u})- (\Corr-\mathcal{C}[\bi{u}_k]) \right] \Vert_{\mathbb{Y}}+\alpha_k \Vert \delta \bi{u} \Vert_{\mathbb{X}},
\end{eqnarray}
with Hilbert spaces $\mathbb{X}$, $\mathbb{Y}$, noise covariance operator $\Gamma$ defined in \eqref{eq:defi_Gamma_n}. The Lagrange function takes the form $\mathcal{L} (\delta \bi{u}_k, \mu):=\Vert (\Gamma_n^{1/2} \times \Gamma_n^{1/2}) \left[\mathcal{C}^{\prime}[\bi{u}_k](\delta \bi{u}_k)- (\Corr-\mathcal{C}[\bi{u}_k]) \right] \Vert_{\mathbb{Y}}+\alpha_k \Vert \delta \bi{u}_k \Vert_{\mathbb{X}}+\langle \mu, \alpha R \delta \bi{u}_k \rangle_{\mathbb{Z}}$ with the Lagrange multiplier $\mu \in \mathbb{Z}$. The saddle point can be found by
\begin{eqnarray}
  \fl  \begin{pmatrix}
\mathcal{C}^{\prime}[\bi{u}_k]^* (\Gamma_n \times \Gamma_n) \mathcal{C}^{\prime}[\bi{u}_k]+\alpha \Id_{\mathbb{X}} & \alpha R^* \\
\alpha R & 0 
\end{pmatrix}
\begin{pmatrix}
\delta \bi{u}_k\\
\mu 
\end{pmatrix}
=\begin{pmatrix}
\mathcal{C}^{\prime}[\bi{u}_k]^* (\Gamma_n \times \Gamma_n) (\Corr-\mathcal{C}[\bi{u}_k])\\
0
\end{pmatrix}.
\end{eqnarray}


\end{document}